\documentclass{amsart}
\usepackage{amsmath}
\usepackage{amssymb}
\usepackage{amsbsy}
\usepackage{amsfonts}
\usepackage{pict2e}
\usepackage{graphicx}
\usepackage{multirow}
\usepackage{epstopdf}
\usepackage{verbatim}
\usepackage{subfigure}
\usepackage{algorithm}
\usepackage{algpseudocode}
\usepackage{amsmath}
\usepackage{amssymb}
\newcommand{\argmin}[1]{\underset{#1}{\mathrm{argmin}}}

\usepackage{color}
\usepackage{varwidth}
\usepackage[normalem]{ulem}

\usepackage[table]{xcolor}

\def\e{{\varepsilon}}

\usepackage{graphicx}
\usepackage{epstopdf}
\usepackage{float}
\usepackage{subfig}
\usepackage{epsfig}
\usepackage{fullpage}

\newtheorem{theorem}{Theorem}[section]

\theoremstyle{definition}

\theoremstyle{remark}
\newtheorem{remark}[theorem]{Remark}

\numberwithin{equation}{section}

\begin{document}

\title{An $L^1$ penalty method for general Obstacle problems}
\author{Giang Tran}
\address{Department of Mathematics, University of California, Los Angeles}
\email{giangtran@math.ucla.edu}

\author{Hayden Schaeffer}
\address{Department of Mathematics, University of California, Irvine}
\email{}

\author{William M. Feldman}
\address{Department of Mathematics, University of California, Los Angeles}
\email{wfeldman10@math.ucla.edu}
\author{Stanley J. Osher}
\address{Department of Mathematics, University of California, Los Angeles}
\email{sjo@math.ucla.edu}

\begin{abstract}
We construct an efficient numerical scheme for solving obstacle problems in divergence form. The numerical method is based on a reformulation of the obstacle in terms of an $L^1$-like penalty on the variational problem. The reformulation is an exact regularizer in the sense that for large (but finite) penalty parameter, we recover the exact solution. Our formulation is applied to classical elliptic obstacle problems as well as some related free boundary problems, for example the two-phase membrane problem and the Hele-Shaw model. One advantage of the proposed method is that the free boundary inherent in the obstacle problem arises naturally in our energy minimization without any need for problem specific or complicated discretization.  In addition, our scheme also works for nonlinear variational inequalities arising from convex minimization problems.
\end{abstract}

\maketitle

\section{Introduction}

In this work we are concerned with the construction and implementation of an unconstrained minimization problem which gives the same solution as a corresponding obstacle problem. The classical obstacle problem models the equilibrium state of an elastic membrane stretched atop of a physical obstacle with fixed boundary conditions. This has a direct mathematical interpretation as an energy minimization (\textit{i.e.} the classical elastic energy of the membrane) with the addition of a constraint (\textit{i.e.} the solutions are bounded below by the obstacle). The general obstacle framework has found applications in steady state fluid interaction, thin-plate fluid dynamics, geometry, elastostatics, etc. 

The original theory for obstacle problems centered around minimizations of the form:
\begin{align*} \min_{u{\in K}}  a(u,u) - \left<f,u\right> ,\end{align*} 
where $a(-,-)$ is a bounded and coercive bilinear form on some Sobolev space $V$, $K = \{ v \geq \varphi\}$ for some smooth $\varphi$, and $<,>$ is the standard $L^2$ inner product \cite{rodrigues1987obstacle,kinderlehrer2000introduction,friedman2010variational}. This minimization problem is equivalent to the problem of finding a $u\in K$ satisfying the variational inequality:
\begin{align*} 
a(u,v-u)  \ \geq \ \left<f,v-u\right> \quad \text{ for all} \quad v \in K,
 \end{align*} 
which can be considered as the Euler-Langrange equation for the constrained problem. In this work, we will use an $L^1$-like penalty on the original variational form: 
\begin{align*}
 \min_{u}  a(u,u) - \left<f,u\right> + \mu \int \max(\varphi-u,0)\, dx,
 \end{align*}
which is an exact penalty for sufficiently large $\mu>0$, see  \cite{friedlander2007exact, mangasarian1985sufficiency}.  
For more details on general theoretical results including regularity of solutions for the obstacle and related free boundary problems, see for example, \cite{caffarelli1998obstacle,caffarelli1976regularity, caffarelli1998obstaclelecture}.

Over the years, there have been many numerical methods for solving various types of obstacle problems. A vast majority of those algorithms use the weak variational inequality characterization to approximate the solutions numerically. For example, in \cite{hoppe1987multigrid}, the authors construct a finite difference scheme based on the variational inequality and use a multigrid algorithm to speed up computations.  In \cite{tai2003rate2} the finite element formulation of the variational inequality is solved using Schwarz domain decomposion method. The convergence of Schwarz domain decomposition for nonlinear variational inequalities is established in \cite{tai2003rate,badea2003convergence}.  Multilevel preconditioners to solve the resulting linear subproblems generated by the finite element discretization was used in \cite{hoppe1994adaptive}. Also, in \cite{weiss2009posteriori} an adaptive finite element method was proposed to solve the variational inequality for one-body contact problems. In another approach \cite{conrad1988approximation}, continuation methods were introduced to approximate the solutions to obstacle problems.  

Alternative approaches use the constrained optimization formulation to construct appropriate algorithms. For example, in \cite{hintermuller2011obstacle}, the constraint is incorporate into the energy via a Langrange multiplier. To solve the resulting saddle point problem a  primal-dual active set algorithm is used.  It should be noted that the existence of solutions to their saddle point problem relies on regularizing the functional, due to the lack of differentiability. A penalty formulation (different from the one  we used here) was proposed in \cite{scholz1984numerical} in order to encourage solutions to satisfy the constraint. However, that method is not exact and requires the penalty parameter to be $\mathcal{O}(h^{-2})$, where $h$ is the grid spacing.

It is also possible to solve the obstacle problem using the complementarity conditions \cite{rodrigues1987obstacle}. With the help of the level set method \cite{osher1988fronts}, the authors of \cite{majava2004level} construct a method to locate the contact set of the obstacle problem. Once the contact set is located, the solution to the obstacle problem can be found directly without the need of the variational inequalities. 

For the two-phase membrane problem, which is a double obstacle problem, the author of \cite{bozorgnia2011numerical} introduces two algorithms. The first method uses finite differences.  The solution is split into two parts, a positive and a negative part, which results in a  coupled system of PDE with matching conditions. In the second method,  a finite element approach is done on a regularized version of the problem so as to avoid the non-differentiablity of the $L^1$-like functions.

Methods and models using $L^1$-like terms are quite common in the fields of imaging science and optimization \cite{candes2006robust,candes2011robust,donoho2006compressed,candes2013phase}. An important aspect of such methods is their efficiency and robustness, which is due in part to the works of \cite{goldstein2009split,cai2009linearized,chambolle2011first}.  Recently, the use of $L^1$-based optimization (and the related low-rank models) has been revived following the early work of \cite{brezis1973ope,brezis1974monotone,brezis1974solutions} and introduced to numerical PDE and computational physics because of its connections to sparsity and compressive sensing. The goal is to construct efficient representation and to create fast solvers for numerical solutions of PDEs. For example,  $L^1$ optimization was used in \cite{schaeffer2013sparse} for multiscale PDEs and in \cite{nelson2013compressive,ozolicnvs2013compressed,ozolinvs2013compressedCPW} for quantum mechanical models. Also, in  \cite{hou2013sparse} an $L^1$ regularized least squares model was constructed to approximate coefficients of a second order ODE whose solutions are associated with intrinsic mode functions. Efficient (sparse) solution representation using low-rank libraries was applied to dynamical systems with bifurcations \cite{brunton2013compressive}. And the use of compressive sensing for fluid dynamics models has seen some recent success, for an example see \cite{bright2013compressive}.

In this work, we connect the $L^1$ based methodology used in imaging science and optimization to obstacle problems and free boundary problems. In particular, we provide some theoretical results on solutions of $L^1$ regularized variational methods to the solutions of obstacle problems with zero obstacle. We derive bounds on the exactness of the penalty formulation as well as construct a fast and simple algorithm to  solve the non-differentiable unconstrained problem. Unlike other penalty methods, we do not require the penalty parameter to go to $\infty$  (for sufficiently smooth obstacles) and no regularization of the penalty is required. 

The outline of this work is as follows. In Section \ref{sec: elliptic case}, we relate $L^1$ optimization to various obstacle problems. We review the obstacle problem formation in Section \ref{sec:obstacle}, and derive a concrete bound for our penalty parameter. In Section \ref{sec:free boundary}, we show how to construct an obstacle problem from a class of free boundary problems. The numerical method and results are described in Sections \ref{sec:numerical} and \ref{sec:simulation}, respectively.  Concluding remarks are given in Section \ref{sec:conclusion}.


 \section{Motivation}\label{sec: elliptic case}
In this section, we motivate the use of $L^1$ based optimization for obstacle problems by establishing a connection between solutions of an $L^1$ penalized variational method and the solutions of obstacle problems with zero obstacle. These problems were considered in \cite{brezis1974solutions, caflisch2013pdes} and can be used for finding compactly supported functions.  Given $f \in L^{2}(\mathbb{R}^d)$, consider the following functional defined for $v \in H^1(\mathbb{R}^d) \cap L^1(\mathbb{R}^d)$,
\begin{equation}
 \mathcal{J}(v) = \int_{\mathbb{R}^d} \frac{1}{2}|\nabla v|^2-fv+\mu|v| \, dx,
 \label{eqn:l1min}
\end{equation}
Then for all test functions $\psi \in H^1(\mathbb{R}^d)\cap L^1(\mathbb{R}^d)$, its unique minimizer $u$,
\begin{equation}
 u = \argmin {}\{ \mathcal{J}(v)\mid v \in H^1(\mathbb{R}^d) \cap L^1(\mathbb{R}^d)\},
 \label{eqn:u0}
 \end{equation}
satisfies the following equation:
\begin{equation}
\int_{\mathbb{R}^d} \nabla u \cdot \nabla \psi - f\psi+\mu p(u) \psi \, dx = 0,
\label{eqn:u}
\end{equation}
where $p(u)$ is an element of the subdifferential of the $L^1$ term in Equation~\eqref{eqn:l1min} and can be identified by (see \cite{caflisch2013pdes}): 
\begin{equation*}
p(u) =
\left\{
\begin{array}{lll}
\text{sign} (u)  & \hbox{ if } & u\neq  0 \\
-f/\mu & \hbox{ if } &  u = 0.
\end{array}\right.
\end{equation*}

We also consider the solution of the following obstacle problem, 
\begin{equation}\label{eqn: obstacle min}
 \bar{u} = \argmin{} \{ \mathcal{J}(v)\mid v \in H^1(\mathbb{R}^d) \cap L^1(\mathbb{R}^d) \ \hbox{ and } \ v \geq 0\}.
  \end{equation}
As a minimizer, $\bar{u}$ satisfies the variational inequality 
\begin{equation}
 \int_{R^d} \nabla \bar{u} \cdot \nabla \psi - f\psi+\mu \psi \, dx \geq 0,
 \label{eqn:ubar}
 \end{equation}
 for all test functions $\psi \in H^1(\mathbb{R}^d)\cap L^1(\mathbb{R}^d)$ with $\psi \geq 0$. One can analogously define $\underline{u}$ as the minimizer of $\mathcal{J}$ over $v \in H^1(\mathbb{R}^d) \cap L^1(\mathbb{R}^d)$ with $v \leq 0$. We will show that the solutions to the variational problems above, $u$, $\bar u$ and $\underline u$, are related. For the rest of the paper, we denote 
  \begin{align*}u_+ :=\max (u,0) \quad \text{and} \quad u_- :=\min (u,0). \end{align*}
\begin{theorem}\label{thm: obstacle l1}
Let $u$ and $\bar u$ be the solutions of Equations~\eqref{eqn:u0} and \eqref{eqn: obstacle min}, respectively, then $\bar{u} = u_+ $. Moreover, if $f\geq 0$  then $\bar u = u$. Similarly, we have $\underline{u} = u_- $ and if $f\leq0$ then $\underline{u} = u$.
\end{theorem}
\begin{proof}
Let $w \geq 0 $ be a solution of the variational inequality:
\begin{equation}
\int_{\mathbb{R}^d} \nabla w \cdot \nabla \psi - f\psi+\mu \psi \, dx \geq 0,
\label{eqn:supersoln2}
\end{equation}
for all $\psi \in H^1(\mathbb{R}^d)\cap L^1(\mathbb{R}^d)$ with $\psi \geq 0$. Next, since $(u-w)_+$ is a valid test function for Equation~\eqref{eqn:u} the following holds:
\begin{align*}
0 &= \int_{\mathbb{R}^d} \nabla u \cdot \nabla(u-w)_+-f(u-w)_++\mu p(u) (u-w)_+ \, dx.
\end{align*}
Since  $p(u) = 1$ on $\{ (u-w)_+ \neq 0 \} \subset \{ u>0\}$, we have
\begin{align*}
0 &=  \int_{\mathbb{R}^d} \nabla (u-w) \cdot \nabla(u-w)_++\nabla w \cdot \nabla(u-w)_+-f(u-w)_++\mu (u-w)_+\,  dx .
\end{align*}
Note that $(u-w)_+$ is also a valid test function for Equation~\eqref{eqn:supersoln2}, so the sum of the last three terms in the above equation is non-negative. Therefore
\begin{align*}
0 \geq\int_{\mathbb{R}^d}\nabla (u-w) \cdot \nabla(u-w)_+ \, dx= \int_{\mathbb{R}^d} |\nabla (u-w)_+|^2 \, dx.
\end{align*}
Thus $(u-w)_+ = c$ a.e., for some non-negative constant $c$. Since $(u-w)_+\in L^1(\mathbb{R}^d)$, we have $c=0$, which means $u_+\leq w$. In particular, since $\bar u$ is also a supersolution of \eqref{eqn:supersoln2}, we have $u_+\leq \bar u$.

As for $\bar u$, after noting that for any $\varepsilon $  the perturbation $\bar u -\varepsilon(\bar u -w)_+$ is an admissible function in the minimization \eqref{eqn: obstacle min}, a similar calculation shows that
$$0\leq\left.\frac{d}{d\e}\right|_{\e = 0}\mathcal{J}(\bar{u}-\e(\bar{u}-w)_+) \leq - \int_{\mathbb{R}^d}|\nabla (\bar{u}-w)_+|^2 dx.$$
Using the same argument as before, we conclude that $\bar u\leq w$. Finally, to prove that $u_+\geq \bar u$, we will show that $u_+$ is also a supersolution of \eqref{eqn:supersoln2}. Indeed, since $-f+\mu p(u_+) \leq -f+\mu$ so as long as $\psi \in H^1(\mathbb{R}^d) \cap L^1(\mathbb{R}^d)$ is nonnegative,
$$  \int_{R^d} \nabla u_+ \cdot \nabla \psi +(- f+\mu) \psi \, dx \geq \int_{\mathbb{R}^d} \nabla u_+ \cdot \nabla \psi - f\psi+\mu p(u_+) \psi \, dx = 0. $$
We have proven that $\bar u = u_+$.  In particular, if $f>0$, one can show that  $u$ is non-negative \cite{brezis1974solutions, caflisch2013pdes}. In this case we have $\bar u = u$. This completes the proof.
\end{proof}

\section{Obstacle Problem}\label{sec:obstacle}

In this section, we recall the classical obstacle problem as well as its penalty formulation which contains an $L^1$-like term. It is shown in \cite{friedlander2007exact, mangasarian1985sufficiency} that if the penalty parameter is large enough, the solution of the penalized problem is identical to the solution of the constrained optimization problem (the obstacle problem in our case).  In addition, we provide a lower bound on the size of the penalty parameter of the unconstrained problem as a function of the obstacle.

Consider the problem of minimizing the Dirichlet energy
\begin{equation}\label{eqn:obstacle0}
 \mathcal{J}(v) = \int_\Omega \frac{1}{2} |\nabla v|^2\, dx,
\end{equation}
among all functions $v$ such that $v-g\in H_0^1(\Omega)$ and $v\geq \varphi$, where $\Omega\subset \mathbb{R}^d$ is bounded and $\varphi:\Omega\rightarrow\mathbb{R}$ is a given smooth function, called the obstacle, which has $\varphi \leq g$ on $\partial \Omega$. Its unique minimizer $\bar u$ satisfies the  complementarity problem \cite{rodrigues1987obstacle}:
\begin{equation*}
-\Delta \bar u \geq 0,\quad \bar u\geq \varphi\, ,\quad (-\Delta \bar u)(\bar u -\varphi)=0, \quad \bar u -g\in H_0^1(\Omega).
\end{equation*}
Let $u_\mu$ be the unique minimizer in $H^1_0(\Omega)$ of
\begin{equation}\label{eqn:mumin}
\mathcal{J}_\mu(v) = \int_\Omega \frac{1}{2} |\nabla v |^2+\mu (\varphi-v)_+ \, dx.
 \end{equation}
In  \cite{friedlander2007exact, mangasarian1985sufficiency}, the authors showed that $u_\mu=\bar u$, for $\mu$ large enough and provided a lower bound for $\mu$ which is the $L^\infty$-norm of any dual optimal multiplier of \eqref{eqn:obstacle0}.  That is if $\mu\geq -\Delta v$, for any dual optimal multiplier $v\geq\varphi$ of the optimization \eqref{eqn:obstacle0}, then $u_\mu =\bar u$. Here we provide a concrete lower bound for $\mu$.

\begin{theorem}\label{thrm_para}
Let $u$ and $u_\mu$ be the optimal minimizers of Equations~\eqref{eqn:obstacle0} and \eqref{eqn:mumin}, respectively. Then for any $\mu$ such that $-\Delta \varphi \leq \mu$ we have $u_\mu = u$.
\end{theorem}

\begin{proof}
For any $v\in H^1_0(\Omega)$, define $w = v+(\varphi-v)_+$. Then $w$ is a valid test function for \eqref{eqn:obstacle0}, i.e., $w \geq \varphi$. Compute
 \begin{align*}
 \mathcal{J}_\mu(w) &= \int_\Omega \frac{1}{2}|\nabla v|^2+\nabla(\varphi-v)_+\cdot \nabla v +\frac{1}{2}|\nabla(\varphi-v)_+|^2 \,  dx \\
 & = \int_\Omega \frac{1}{2}|\nabla v|^2+\nabla(\varphi-v)_+\cdot \nabla\varphi  -\frac{1}{2}|\nabla(\varphi-v)_+|^2\, dx \\
 & \leq \int_\Omega \frac{1}{2}|\nabla v|^2+\mu(\varphi-v)_+ -\int_{\Omega}\frac{1}{2}|\nabla(\varphi-v)_+|^2 \, dx  \\
 & = \mathcal{J}_\mu(v)-\int_{\Omega}\frac{1}{2}|\nabla(\varphi-v)_+|^2\, dx.
 \end{align*}
The inequality in the third line holds since $-\Delta \varphi \leq \mu$ in the weak sense. Therefore, $\mathcal{J}_\mu(v+(\varphi -v)_+) < \mathcal{J}_\mu(v)$ unless $\nabla(\varphi-v)_+$ is zero, which implies  $(\varphi-v)_+=0$ since $(\varphi-v)_+\in H^1_0(\Omega)$.  In particular, we have
 \begin{equation*}
 \mathcal{J}_\mu(u_\mu + (\varphi -u_\mu)_+)\leq \mathcal{J}_\mu(u_\mu).
 \end{equation*}
Since $u_\mu$ is the uniqueness minimizer of \eqref{eqn:mumin},  $(\varphi-u_\mu)_+ =0$ which means $u_\mu\geq \varphi$ is a valid test function for \eqref{eqn:obstacle0}. In addition, we observe
  $$\mathcal{J}(u_\mu) = \mathcal{J}_\mu(u_\mu) \leq \mathcal{J}_\mu(u) =\mathcal{J}(u).$$
 Since $u$ is the unique minimizer of $\mathcal{J}$, we conclude that $u = u_\mu$.

 \end{proof}
\begin{remark} It is worth noting that in the numerical experiments provided in this work, the smaller the value of $\mu$ is, the faster the iterative scheme converges to the steady state. Therefore, an explicit lower bound on $\mu$ greatly improves the convergence rate of the method. 
\end{remark}
\section{Free boundary problem}\label{sec:free boundary}
In this section, we show how to put a class of free boundary problems into a form where the methodology of Sections \ref{sec: elliptic case} and \ref{sec:obstacle} can be directly applied.  We emphasize that for these problems our primary interest is in the location of the free boundary $\partial\{u>0\}$ as opposed to the solution itself.  For a concrete example, we will focus our attention on the Hele-Shaw model.

\subsection{Turning a Class of Free Boundary into an Obstacle}
\label{sec:obstacleExact}
 Consider the solution of the following free boundary problem in $\mathbb{R}^d$:
\begin{equation}
\left\{
\begin{array}{lll}
-\Delta u = f-\gamma & \hbox { in } & \{ u >0 \} \\
u = |\nabla u| = 0 & \hbox { on } & \partial \{ u>0\},
\end{array}\right.
\label{eqn:fbp0}
\end{equation}
with some given source function $f$ and constant $\gamma$. In this form, we can see the connection to an $L^1$-minimization problem (Equation~\eqref{eqn:l1min} with $\gamma = \mu$). In general, this can be difficult to solve numerically because of the free boundary $\partial\{u>0\}$. We will show that our method naturally treats the free boundary conditions thereby avoiding any difficulty in directly tracking or approximating it. The details are described below. 

First let us define the obstacle:
$$ \varphi := -\tfrac{\gamma}{2d} |x|^2- (-\Delta)^{-1} f(x).$$
Then  the function $w = u+\varphi$ will be the least super harmonic majorant of $\varphi$ in $\mathbb{R}^d$, that is,  it solves the free boundary problem:
\begin{equation}
\left\{ 
\begin{array}{lll}
-\Delta w = 0 & \hbox { in } & \{ w > \varphi \} \\
 \nabla w = \nabla \varphi & \hbox { on } & \partial \{ w>\varphi\}.
\end{array}\right.
\label{obstacleExact}
\end{equation}
By transforming the PDE \eqref{eqn:fbp0}, we replace the source term with an obstacle. Indeed, the solution $w$ of Equation~\eqref{obstacleExact} is the unique minimizer of the following optimization problem:
\begin{equation}
w =\argmin{v\in H^1(\mathbb{R}^d)} \int \frac{1}{2}|\nabla v|^2 + \mu (\varphi-v)_+\, dx,
\label{eqn:wtran}
 \end{equation}
 for some parameter $\mu$.  Therefore, by finding the solution to the unconstrained optimization problem (Equation \eqref{eqn:wtran}), we can locate the free boundary to the original problem directly.
\subsection{Hele-Shaw}\label{sec:Hele-Shaw}
Let us recall the classical Hele-Shaw problem with a free boundary. Let $K \subset \mathbb{R}^d$ be a compact set and $\Omega_0 \supset K$ be open and bounded. Fluid initially occupies $\Omega_0$ and is injected at a constant rate 1 per unit length through the surface $K$. The fluid expands and occupies the region $\Omega_t$ with the free boundary $\Gamma_t$. Let $p(x,t):\mathbb{R}^d\times [0,\infty)\rightarrow\mathbb{R} $ be the pressure of the fluid. For simplicity we consider a slight variant of the Hele-Shaw model where $p(x,t)$ instead of its normal derivative is equal to $1$ on $\partial K$, see \cite{jerison2005one}.  Then the time integral of $p$, $u(x,t) =  \int_0^t p(x,\tau)d\tau$ formally satisfies (see also \cite{elliott1981variational, gustafsson1985applications}),
\begin{equation}\label{eqn: obs Hele-Shaw}
\left\{
\begin{array}{lll}
-\Delta u = \chi_{\Omega_0}-1 & \hbox{ in } & \Omega_t(u) \setminus K \\
u = t & \hbox{ on } & K \\
u = |\nabla u| = 0 & \hbox{ on } & \Gamma_t(u).
\end{array}\right.
\end{equation}
Note that $\Omega_t(u) = \{u>0\}$. We are free to solve Equation~\eqref{eqn: obs Hele-Shaw} only since the free boundary is the same as the free boundary of the pressure $p$. Here we consider the stable flow examples. For an example of a numerical method to solve the unstable Hele-Shaw flow (with the known fingering effect), see \cite{hou1997hybrid}. 

Let us define the obstacle:
$$ \varphi_0 := -\tfrac{1}{2d}|x|^2-(-\Delta)^{-1}\chi_{\Omega_0}.$$
Similar to Section \ref{sec:obstacleExact}, the function $w = u + \varphi_0$ solves
$$w = \argmin{v\in \mathcal{V}_t} \int_{\mathbb{R}^d \setminus K} \frac{1}{2}|\nabla v|^2+\mu(\varphi_0 - v)_+ \, dx, $$
where the admissible set is defined as follows
$$ \mathcal{V}_t = \{ (v-\varphi_0) \in (H^1 \cap L^1)(\mathbb{R}^d \setminus K) :  v = \varphi_0+t  \ \hbox{ on } \ \partial K\}.$$
For computational purposes, it is desirable to avoid solving a minimization problem in a possibly complicated domain $\mathbb{R}^d \setminus K$. So the boundary condition can be included as a secondary obstacle. To do so, we define the new obstacle,
\begin{equation*}  \varphi = \varphi_0+ t \chi_K,\label{eqn:varphi}\end{equation*}
with the associated double penalized energy,
\begin{equation}
\mathcal{J}_{\gamma}(v) = \int_{\mathbb{R}^d } \frac{1}{2}|\nabla v|^2+\gamma_1(\varphi - v)_+-\gamma_2 (t\chi_K - v)_- \, dx,
\label{eqn:HSgeneral}
\end{equation}
for some parameters $\gamma_1$ and $\gamma_2$.

Since $\varphi$ is not smooth, the argument of the previous section, namely that $-\Delta \varphi - \gamma_1$ is subharmonic for $\gamma_1$ sufficiently large, is not directly applied. However, we can build a smooth approximation for the obstacle using mollifier.   On the other hand (at least heuristically), when one minimizes a discretization of $\mathcal{J}_\gamma$ with grid spacing $h$, the minimizer of the discretization is as good an approximation to $\mathcal{J}_\gamma$ as it is to,
\begin{equation*}
\mathcal{J}^h_{\gamma}(v) = \int_{\mathbb{R}^d} |\nabla v|^2+\mu\max\{t \rho_h * \chi_K+\varphi_0 - v,0\}  \, dx,
\label{eqn:obsMollifier}
\end{equation*}
where $\rho_h = h^{-d}\rho(h^{-1}x)$ with $\rho\in C^\infty(\mathbb{R}^d)$ being a standard mollifier.  Now one can estimate:
$$ \|\Delta \rho_h * \chi_K\|_{L^\infty} \leq \|\Delta \rho_h\|_{L^1}\|\chi_K\|_{L^\infty} \leq h^{-2}\|\Delta \rho\|_{L^1(\mathbb{R}^d)}.$$
In particular by the result of Section \ref{sec:obstacle}, for the mollified functional $\mathcal{J}^h_\gamma$ as long as
$$\mu \geq th^{-2}\|\Delta \rho\|_{L^1(\mathbb{R}^d)}+1,$$
 the global minimizer  solves the obstacle problem with $t \rho_h * \chi_K+\varphi_0 - s$ as the obstacle.
\begin{remark}
The solution of Equation \eqref{eqn: obs Hele-Shaw} can also be viewed as the minimizer of
$$ \mathcal{J}(v) = \int_{\mathbb{R}^d \setminus K}\frac{1}{2} |\nabla v |^2-\chi_{\Omega_0}v+|v| \, dx, $$
over the admissible set,
$$ \mathcal{V}_t = \{ v \in (H^1 \cap L^1)(\mathbb{R}^d \setminus K):  v = t  \ \hbox{ on } \ \partial K\}.$$
Let us call
$$ \tilde{u}(\cdot,t) = \argmin{} \{ \mathcal{J}(v) : v \in \mathcal{V}_t \},$$
then as in the Section \ref{sec: elliptic case}, $\tilde{u}$ will be the same as the solution of the obstacle problem \eqref{eqn: obs Hele-Shaw} obtained as the infimal non-negative supersolution. Simulations based on this observation yield similar results to those of the penalized energy.
\end{remark}
\section{Numerical Method}\label{sec:numerical}
For the numerical method, we employ the energy minimization formulation and discretize the energy using a uniform fixed grid. The energies are convex, so we construct an algorithm via \cite{goldstein2009split} to decouple the problem into an explicit part and a strictly convex part. In the explicit part, the optimal value can be computed directly using shrink-like operators. For the strictly convex part, we can use either a conjugate gradient method or an accelerated gradient descent method to quickly solve the subproblem. The detailed algorithm and its construction are described here. Note that for each problem, there could be slight variations in the algorithm, which we will explain in each subsection.

Consider the following discrete energy (where $\nabla_h u$ is a discretization of the gradient of $u$)
\begin{equation*}
 \min_{u\geq \varphi}  F(\nabla_h u),
  \end{equation*}
 where $F$ is a convex functional and $\varphi$ is a given function. To solve this problem, we first convert it into an unconstrained problem by using the penalty method:
 \begin{equation*}
 \min_{u} F(\nabla_h u) + \mu \,|(\varphi-u)_+|,
 \end{equation*}
for some parameter $\mu>0$. Here $|\cdot |$ denotes the $L^1$-norm.
Since $(\cdot)_+$ is not differentiable, we construct an equivalent minimization problem using \cite{goldstein2009split}. We first introduce an auxiliary variable $v=\varphi-u$ then apply the Bregman iteration:
\begin{equation*}
 \begin{cases}
 (u^{k+1}, v^{k+1}) &= \argmin{u,v} \, F(\nabla_h u) + \mu\,|v_+| +\frac{\lambda}{2} ||v-\varphi+u+b^k||_2^2,\\
 \quad\quad b^{k+1} & =  b^k  + u^{k+1} + v^{k+1} -\varphi.
 \end{cases}
 \end{equation*}
 Now we can efficiently solve the minimization by splitting it into two subproblems with respect to $u$ and $v$:
 \begin{equation*}
 \begin{cases}
 \text{Step 1:}\quad u^{n+1} & =\argmin{u}\,\mathcal{F}(u) = F(\nabla_h u) +  \frac{\lambda}{2} ||v^n-\varphi+u+b^n||_2^2,\\
 \text{Step 2:}\quad v^{n+1} & = \argmin{v}\, \mu\,|v_+| +\frac{\lambda}{2} ||v-\varphi+u^{n+1}+b^n||_2^2.
 \end{cases}
 \end{equation*}
  The solution for $v$ is given explicitly:
 \begin{align*}
v= S_{+} \left(\varphi -u^{n+1} -b^n, \ \frac{\mu}{\lambda}\right),
 \end{align*}
 where $S_{+} (z, c) := (z-c)$ if $z>c$, $z$ if $z<0$, and $0$ otherwise. 

 To solve the $u$ subproblem, we consider two cases for the first variation, $G$, of $F$. If $G$ is linear, for example taking $F(\nabla u)=\dfrac{1}{2}\int\nabla u \cdot A \nabla u\,dx$ and $A$ is positive semi-definite, then the first variation is:
 \begin{align*}
 ( \lambda I - \nabla \cdot A\nabla) u=\lambda (\varphi-v^n-b^n),
 \end{align*}
 which can be solved by using the conjugate gradient method. In the case where $G$ is non-linear, for example taking $F(\nabla u)=\int \sqrt{1+ |\nabla  u|^2}\, dx$, we leverage the strict convexity of the functional to quickly solve the substep by using Nesterov's acceleration method \cite{nesterov2004introductory}. The resulting scheme for $u$ is as follows:
 \begin{align}
 \begin{cases}
 w^{k} &= U^{k} + \frac{\sqrt{L}-\sqrt{\lambda}}{\sqrt{L}+\sqrt{\lambda}} \ (U^{k}-U^{k-1})\\
 U^{k+1} &= w^{k} -\tau (G(w^{k})+ \lambda (v^n-\varphi+w^k+b^n) ),
 \end{cases}
 \label{nest_alg}
 \end{align}
 where $\tau>0$ is a psuedo-time step, $L$ is the Lipschitz norm of $\mathcal{F}$, and $w$ is an auxiliary variable. This scheme has the following convergence bound:
 $$\mathcal{F}(U^k)-\mathcal{F}(U^*) \leq 2 \left(1-\sqrt{\frac{\lambda}{L}}\right)^k \left(\mathcal{F}(U^0)-\mathcal{F}(U^*)\right),$$
 where $u^n=U^0$, $u^{n+1}=U^*$, and $U^*$ is the steady state solution of Equation~\eqref{nest_alg}. Both algorithms are summarized below.

\medskip

\noindent\fbox{%
\begin{minipage}{\dimexpr\linewidth-2\fboxsep-2\fboxrule\relax}
\begin{algorithmic}
\State \underline{\textbf{Algorithm (Linear)}}
\State Given: $u^0,b^0, tol $ and parameters $\lambda,\mu$
\medskip
\While{$||u^{n}-u^{n-1}||_{\infty}>tol$}
\medskip
\State $u^{n+1}=( I - \lambda^{-1 }G)^{-1} (\varphi-v^n-b^n)$
\medskip
\State $v^{n+1}= S_{+} \left(\varphi-u^{n+1}-b^{n},\, \dfrac{\mu}{\lambda}\right)$
\medskip
\State $b^{n+1}=b^n+u^{n+1}+v^{n+1}-\varphi$
\medskip
 \EndWhile
\end{algorithmic}
\end{minipage}%
}

\medskip

\noindent\fbox{%
\begin{minipage}{\dimexpr\linewidth-2\fboxsep-2\fboxrule\relax}
\begin{algorithmic}
\State \underline{\textbf{Algorithm (Non-linear)}}
\State Given: $u^0,b^0, tol $ and parameters $\lambda,\mu$
\medskip
\While{$||u^{n}-u^{n-1}||_{\infty}>tol$}
\medskip
\State $U^{0}=u^{n}$
\medskip
\While{$||U^{k}-U^{k-1}||_{\infty}>tol$}
\medskip
\State $w^{k} = U^{k} + \frac{\sqrt{L}-\sqrt{\lambda}}{\sqrt{L}+\sqrt{\lambda}} \ (U^{k}-U^{k-1})$
\medskip
 \State $U^{k+1} = w^{k} -\tau (G(w^{k})+ \lambda (v^n-\varphi+w^k+b^n) )$
 \medskip
\EndWhile
\medskip
\State $u^{n+1}=U^{*}$
\medskip
\State $v^{n+1}= S_{+} \left(\varphi-u^{n+1}-b^{n},\,\dfrac{\mu}{\lambda}\right)$
\medskip
\State $b^{n+1}=b^n + u^{n+1} + v^{n+1}-\varphi$
\medskip
 \EndWhile
\end{algorithmic}
\end{minipage}%
}

\bigskip

\subsection{Obstacle Problem}\label{sec:obsnum}
Given an obstacle $\varphi : \mathbb{R}^d \rightarrow\mathbb{R}$ and $\mu$ satisfying the condition from Theorem \ref{thrm_para}, we solve the following obstacle problem:
\begin{equation*}
\min_u\int\frac{1}{2}|\nabla u|^2 + \mu (\varphi -u)_+ \, dx.
\end{equation*}
The corresponding discrete problem is given by:
\begin{equation*}
 \min_u  \sum_j \frac{1}{2} |\nabla_h u_j |^2+\mu (\varphi_j-u_j)_+\,  ,
 \end{equation*}
where $u_j=u(x_j)$ and $h>0$ is the uniform grid spacing. Since the functional is quadratic, the Euler-Lagrange equation for the subproblem in terms of $u$ satisfies a Poisson equation:
  \begin{align}
 ( \lambda I- \Delta_h) u=\lambda (\varphi-v+b).
 \label{poisson}
 \end{align}
To approximate the solution of the linear system, we use a few iterations of the conjugate gradient method. It was noted in \cite{goldstein2009split} that full convergence is not necessary within the main iterations, thus we are not required to solve Equation \eqref{poisson} exactly.

\subsection{Two-Phase Membrane Problem}\label{sec:twophasenum}
Consider the following optimization problem arising from finding the equilibrium state of a thin film:
\begin{equation*}
\min_u\int \frac{1}{2}|\nabla u|^2 + \mu_1 u_+ - \mu_2u_- \, dx,
\end{equation*}
for some positive and continuous Lipschitz functions $\mu_1(x)$ and $\mu_2(x)$. The corresponding Euler-Lagrange equation  is:
 \begin{equation*}
 \Delta u = \mu_1 \chi_{\{u>0\}} - \mu_2\chi_{\{u<0\}}.
 \end{equation*}
The regularity of this problem was studied in \cite{shahgholian2007two,petrosyan2012regularity}. Here we are concerned with the numerical approximation of this problem as well as computing the zero level set.

The corresponding discrete minimization problem is given by:
\begin{equation*}
 \min_u  \sum_j \frac{1}{2} |\nabla_h u_j |^2+\mu_1 (u_j)_+-\mu_2(u_j)_-\,.
 \end{equation*}
 Now one can apply the split Bregman method by introducing two auxiliary variables $v_1 = u_+$ and $v_2 = u_-$. However, the algorithm can be further simplified by using the following relations: 
 $$u_+ = \frac{u + |u|}{2},\quad\text{and}\quad u_- = \frac{u - |u|}{2}.$$ 
Now we can rewrite the problem as:
\begin{equation*}
 \min_u  \sum_j \frac{1}{2} |\nabla_h u_j |^2+\alpha u_j+\beta |u_j|,
 \end{equation*}
 where $\alpha = \frac{\mu_1 - \mu_2}{2}$ and $\beta = \frac{\mu_1 + \mu_2}{2}$. In this form, we have a slightly different numerical scheme. As before, the splitting leads to:
 \begin{equation*}
 \min_u  \sum_j \frac{1}{2} |\nabla_h u_j |^2+\alpha u_j+\beta |v_j| + \frac{\lambda}{2} |v_j-u_j-b_j|^2,
 \end{equation*}
and iterative scheme is written as follows
 \begin{align*}
 u^{n+1}&=( \lambda I - \Delta_h)^{-1} ( \lambda (v^n-b^n) - \alpha)\\
 v^{n+1}&=S \left(u^{n+1} + b^n,\, \frac{\beta}{\lambda} \right)\\
b^{n+1}&=b^n + u^{n+1} -v^{n+1},
 \end{align*}
 where the shrink function is defined as $S (z, c) := (|z|-c)_+ \, \text{sign}(z)$.

\subsection{Hele-Shaw}\label{sec:heleshawnum}
As described in Section \ref{sec:Hele-Shaw}, we minimize the obstacle problem transformation of the Hele-Shaw flow:
\begin{equation*}
\min  \int \frac{1}{2}|\nabla u|^2+\gamma_1(\varphi - u)_+-\gamma_2 (t\chi_K - u)_- \, dx,
\end{equation*}
 with $\varphi$ given by Equation~\eqref{eqn:varphi}. The corresponding discretization problem is:
\begin{equation*}
\min_u \sum_j \frac{1}{2} |\nabla_h u_j|^2+\gamma_1(\varphi_j - u_j)_+ - \gamma_2 (t\,\chi_{K,j} - u_j)_-\, .
\end{equation*}
Once again, we construct an equivalent minimization problem by introducing two auxiliary variables $v_1 = \varphi -u$ and $v_2 = u-t\chi_K$. For convenience, we drop the subscript $j$ in all terms:
\begin{equation*}
\min_{u,v_1,v_2,b_1,b_2} \sum_j \frac{1}{2} |\nabla_h u|^2+\gamma_1(v_{1})_+ + \gamma_2 (v_{2})_+ + \frac{\lambda_1}{2}(v_{1} -\varphi + u +b_{1} )^2 + \frac{\lambda_2}{2}(v_{2} - u+ t\chi_{K} + b_{2})^2.
\end{equation*}
The iterative scheme is written as follows:
\begin{align*}
u^{n+1}&= \left( (\lambda_1 + \lambda_2) I - \Delta_h \right )^{-1} \left ( \lambda_1 (\varphi -v_1^n-b_1^n) + \lambda_2(v_2 + t\chi_K + b_2)\right)\\
 v_1^{n+1}&=S_+\left(\varphi - u^{n+1} - b_1^n,\, \frac{\gamma_1}{\lambda_1} \right)\\
  v_2^{n+1}&=S_+\left(u^{n+1} - t\chi_K-  b_2^n,\, \frac{\gamma_2}{\lambda_2} \right)\\
b_1^{n+1}&=b_1^n + v_1^{n+1} - \varphi + u^{n+1}\\
b_2^{n+1}&=b_2^n + v_2^{n+1} - u^{n+1} + t\chi_K.
\end{align*}
Each substep is either a linear system of equations or an explicit update using the shrink-like operators, making it easy to solve.

\section{Computational Simulations}\label{sec:simulation}
In this section, we apply the methods from Section \ref{sec:numerical} to various examples. The iterative schemes stop when the difference between two consecutive iterations in the $L^\infty$ norm is less than a set tolerance, $tol$. We will specify the tolerance parameter for each problem. 
\subsection{Obstacle problem}
For our first examples, we show some numerical results for the minimization problem:
\begin{equation*}
 \min_u\int\frac{1}{2}|\nabla u|^2 + \mu (\varphi -u)_+ \, dx,
  \end{equation*}
with different types of obstacles. In particular, consider the following 1D obstacles:
\begin{equation}
\varphi_1 (x) :=\begin{cases}
				100x^2 &\quad\text{for}\quad 0\leq x\leq 0.25\\
				100x(1-x) -12.5 & \quad\text{for}\quad 0.25\leq x\leq 0.5\\
				 \varphi_1(1-x) &\quad\text{for}\quad 0.5\leq x\leq 1.0,
	             \end{cases}
	             \label{obstacle1d11}
	 \end{equation}
	 and
	\begin{equation}
\varphi_2 (x) :=\begin{cases}
		        10\sin(2\pi x) & \quad\text{for}\quad 0\leq x\leq 0.25\\
		        5\cos(\pi(4x-1)) +5 &\quad\text{for}\quad 0.25\leq x\leq 0.5\\
		        \varphi_2(1-x) &\quad\text{for}\quad 0.5\leq x\leq 1.0.
	        \end{cases}
	        \label{obstacle1d12}
\end{equation}
In both cases, the parameter $\mu$ is determined discretely (see Theorem \ref{thrm_para}) and $u$ is initialized using the obstacle, \textit{i.e.} $u^0 =\varphi$. The results are shown in Fig. \ref{Obstacle1d}. In both cases the numerical solutions are linear away from their corresponding obstacles, which agrees with the analytic solutions:

\begin{figure}[t!]
\centering
\includegraphics[width = 2.3 in]{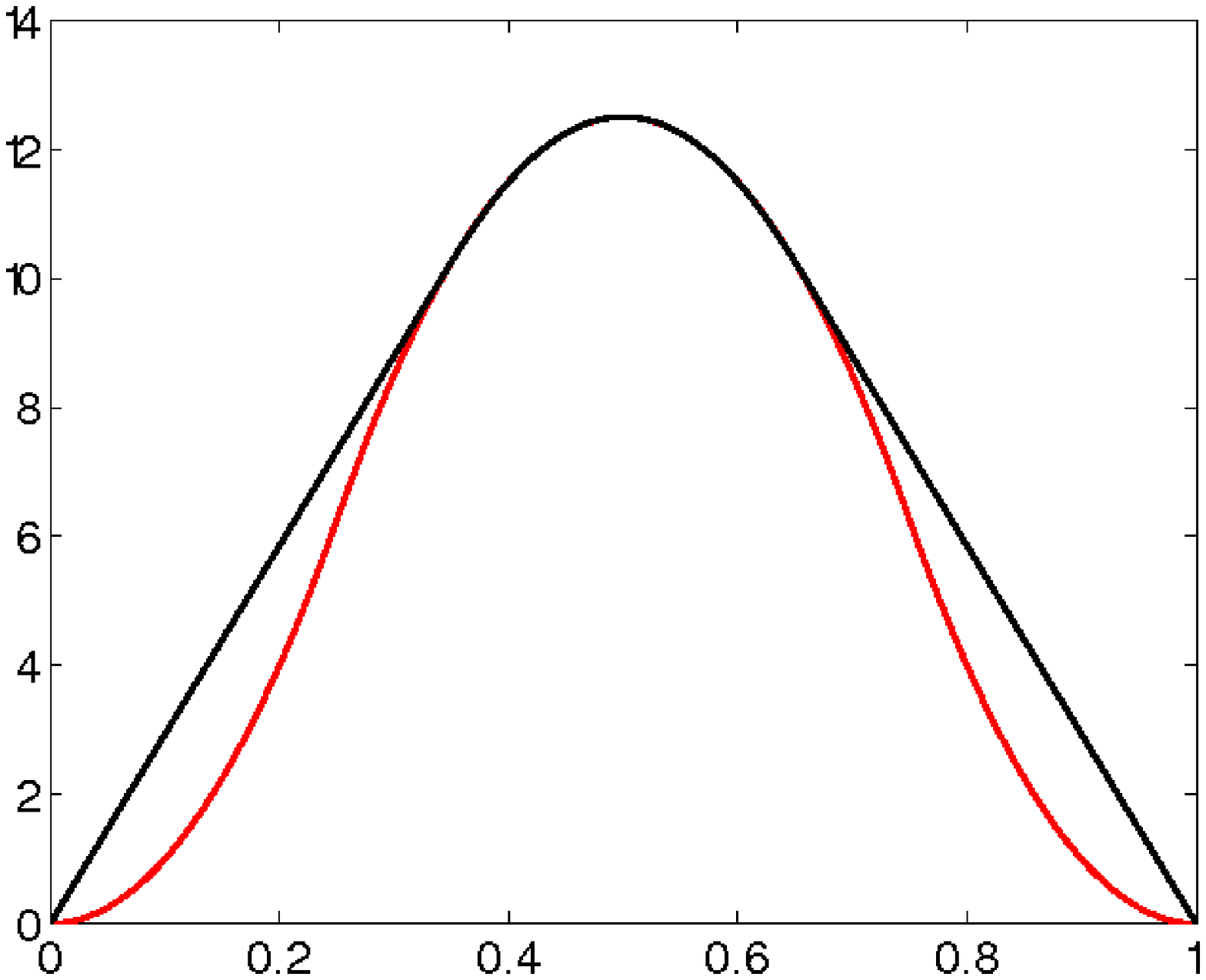}\hspace{0.3cm}
\includegraphics[width = 2.3 in]{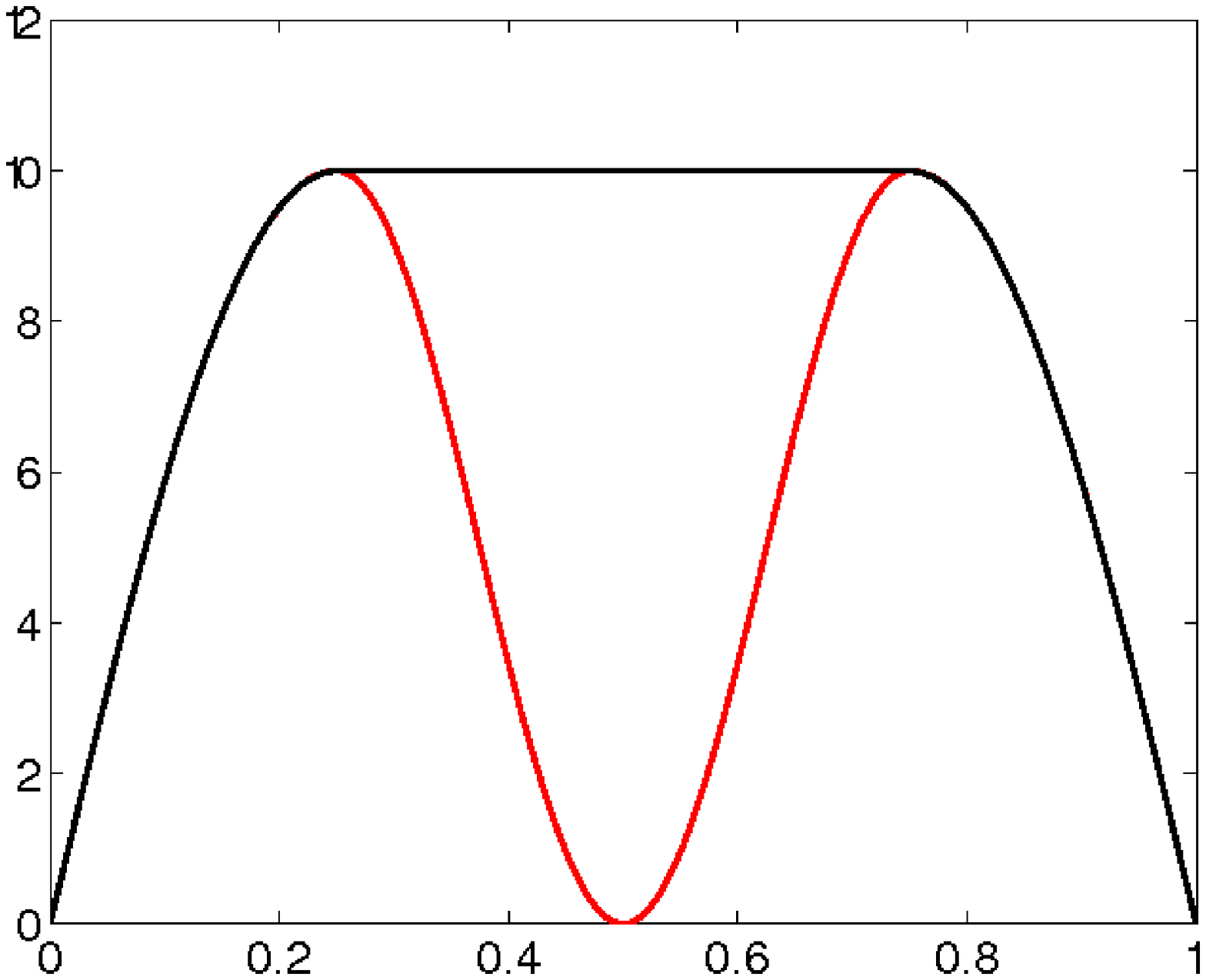}
\caption{The red curves are the obstacles and the black ones are our numerical solutions associated with Equations~\eqref{obstacle1d11} (left) and \eqref{obstacle1d12} (right) after 50 iterations. The grid size is 256,  the parameters are $(\mu,\lambda) = (300,45)$ and $(2.5\times 10^4,250)$, respectively.}
\label{Obstacle1d}
\end{figure}
\begin{figure}[t!]
\centering
\includegraphics[width = 2.3 in]{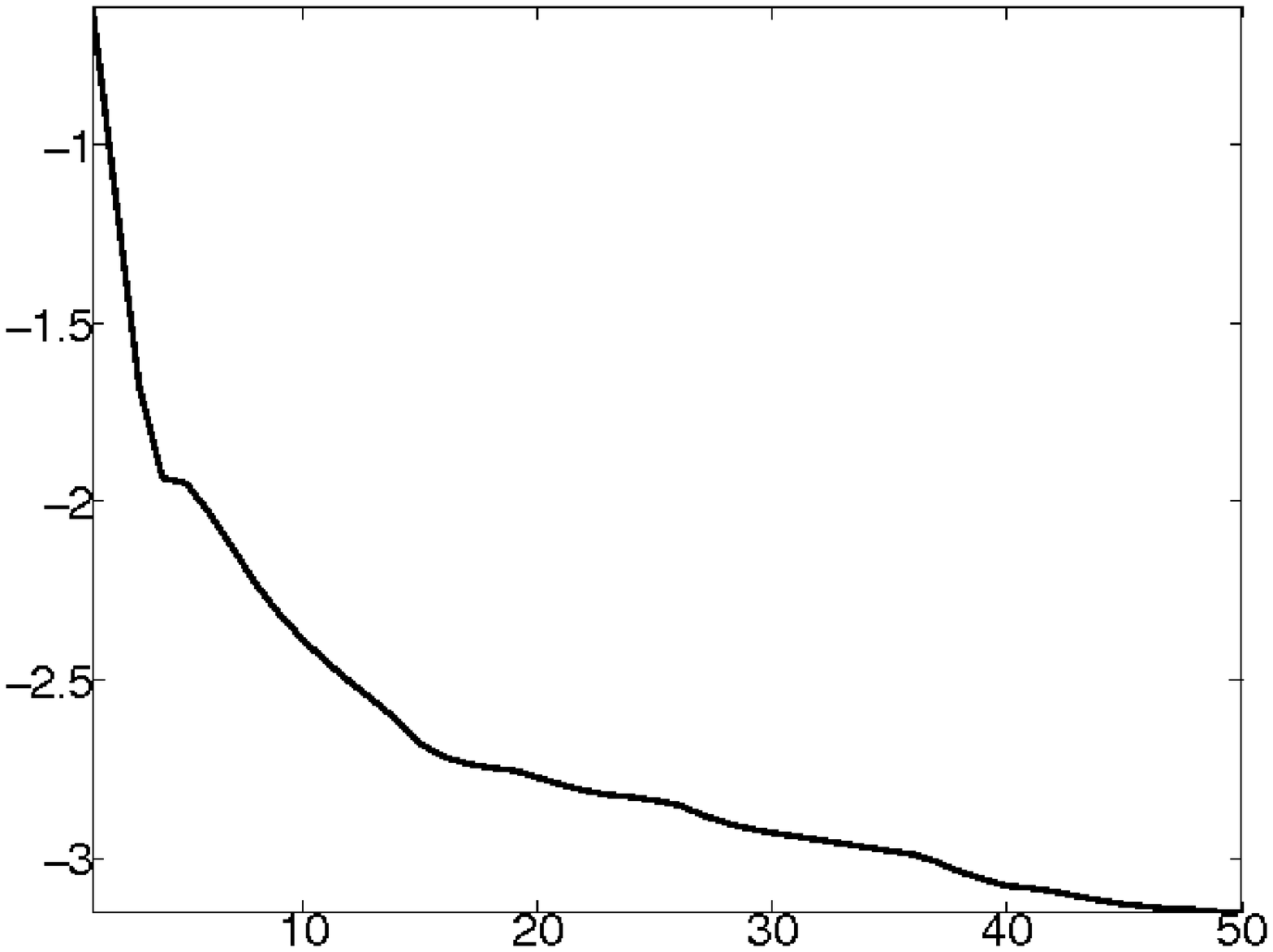}\hspace{0.3cm}
\includegraphics[width = 2.3 in]{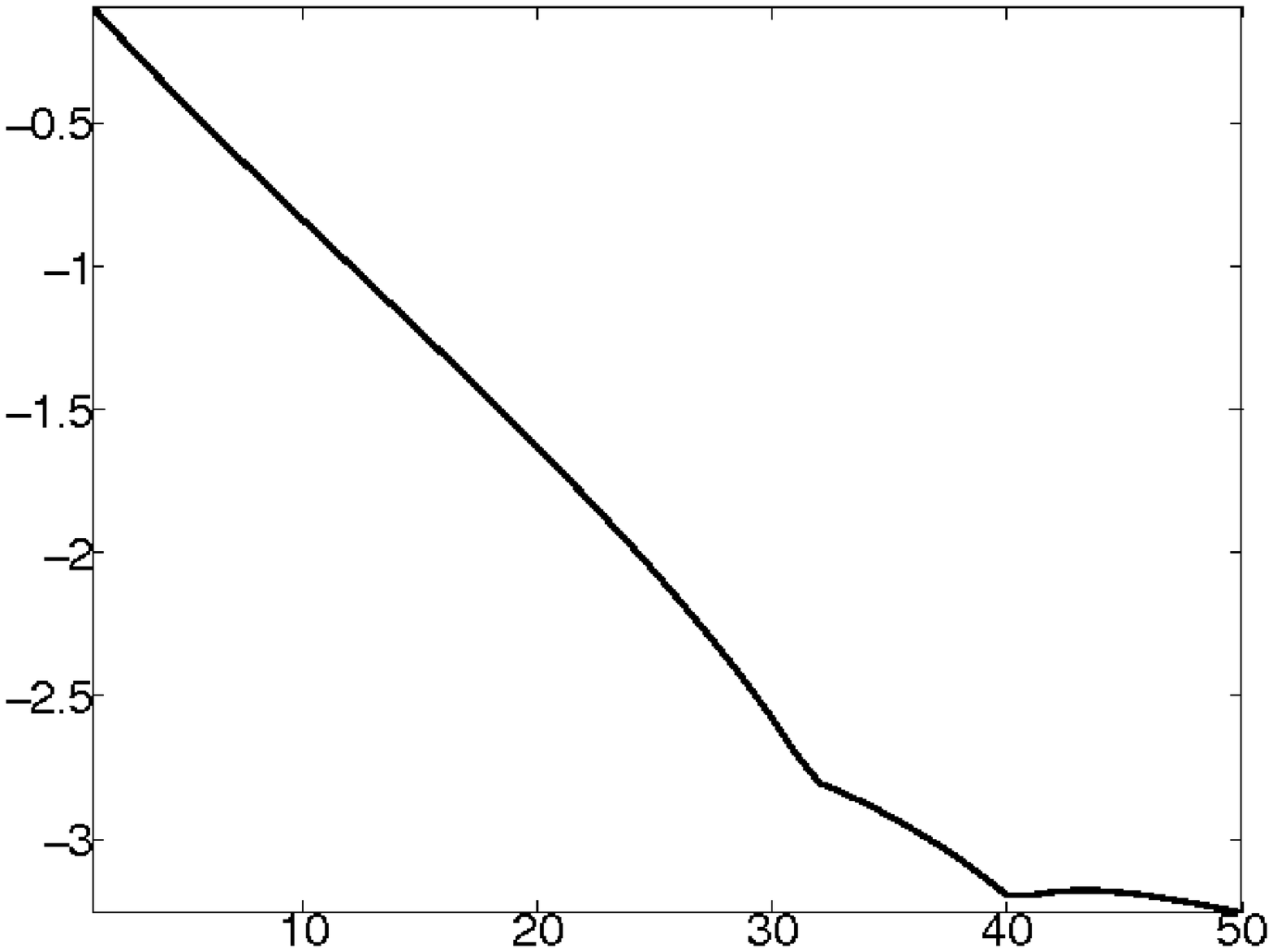}
\caption{The plots correspond to the log relative errors between our numerical solution (from Fig. \ref{Obstacle1d}) and the analytic solution versus number of iterations. The error is measured in the $L^\infty$ norm.}
\label{Obstacle1derror}
\end{figure}

\begin{equation*}
u_{1,exact}(x) = \begin{cases}
(100-50\sqrt{2}) x & \quad \text{for}\quad 0\leq x\leq \frac{1}{2\sqrt{2}}\\
100x(1-x) - 12.5 &\quad\text{for}\quad \frac{1}{2\sqrt{2}} \leq x\leq 0.5\\
u_{1,exact}(1-x) &\quad\text{for}\quad 0.5\leq x\leq 1.0,
\end{cases}
\end{equation*}
and
\begin{equation*}
u_{2,exact}(x) =\begin{cases}
10\sin (2\pi x) &\quad\text{for}\quad 0\leq x\leq 0.25\\
10 &\quad\text{for}\quad 0.25\leq x\leq 0.5\\
u_{2,exact}(1-x)&\quad\text{for} \quad 0.5\leq x\leq 1.0.
\end{cases}
\end{equation*}
The errors between the analytic solutions and the numerical solutions versus the number of iterations associated to obstacle problems \eqref{obstacle1d11} and \eqref{obstacle1d12} are shown in Fig. \ref{Obstacle1derror}. Notice that the numerical scheme has nearly exponential error decay in the beginning.

 Next, we consider a 2D problem on the domain $\Omega = [-2,2]\times [-2,2]$ with the following obstacle:
\begin{equation}
\varphi(x,y) =\begin{cases}
\sqrt{1-x^2-y^2},&\quad\text{for}\quad x^2+y^2\leq 1\\
-1,&\quad\text{otherwise.}
\label{obstacle2dexact}
\end{cases}
\end{equation}
Since the obstacle is radial symmetric, the analytical solution can be solved directly:
\begin{equation*}
u(x,y) = \begin{cases}
\sqrt{1-x^2-y^2},&\quad\text{for}\quad r\leq r^*\\
-(r^*)^2\log (r/2)/\sqrt{1-(r*)^2},&\quad\text{for}\quad r\geq r^*,
\end{cases}
\end{equation*}
where $r = \sqrt{x^2+y^2}$, and $r^*$ is the solution of
\begin{equation*}
(r^*)^2(1-\log (r^*/2)) =1.
\end{equation*}
Our numerical solution and the difference with the analytic solution are presented in Fig. \ref{fig:obstacle2dexact}. For comparison see \cite{majava2004level}.  We can see that the error is concentrated along the contact set, where the function is no longer $C^2$, and is relatively small everywhere else.
\begin{figure}[t!]
\includegraphics[width = 2.5 in]{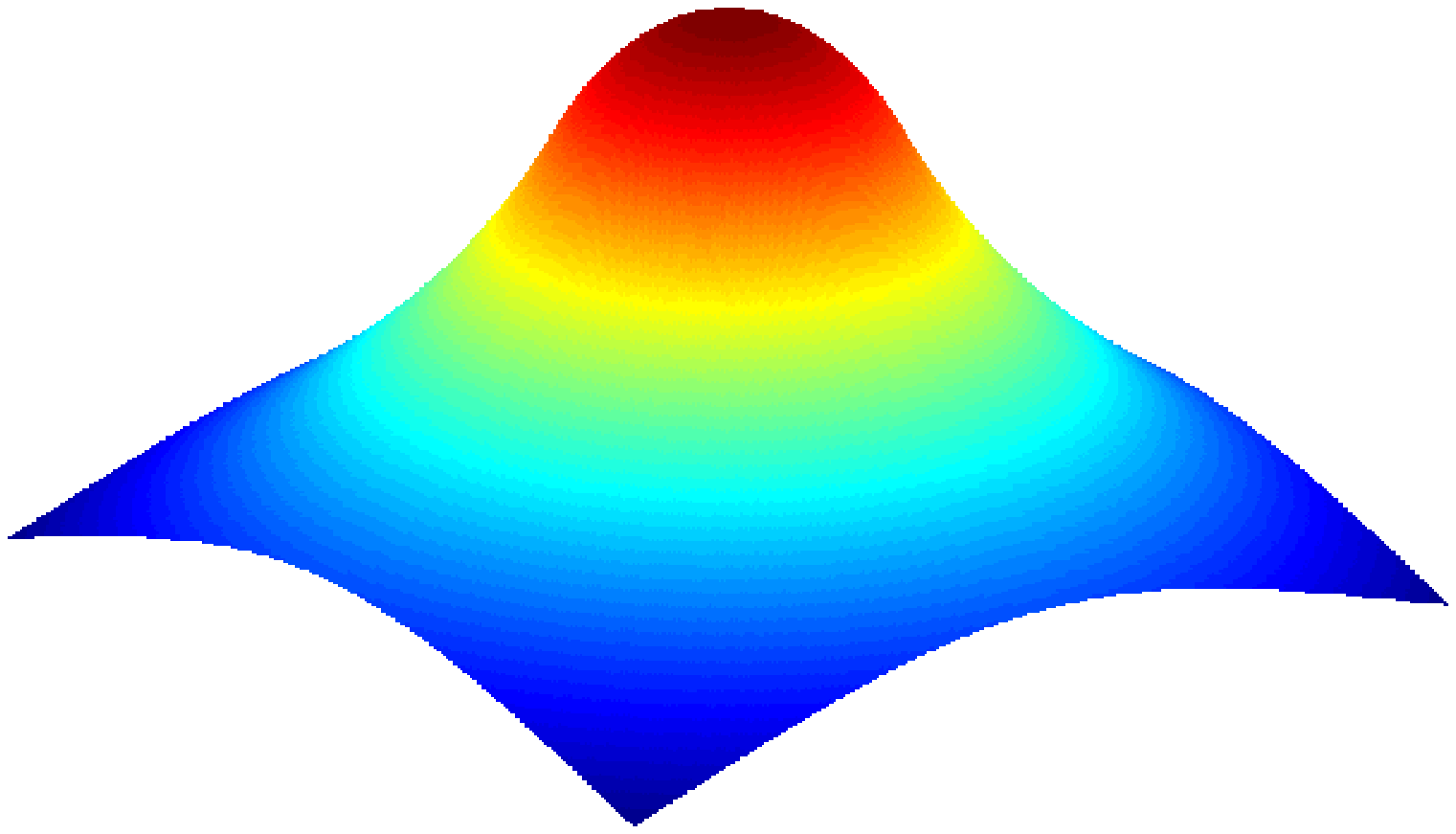}
\includegraphics[width = 2.5 in]{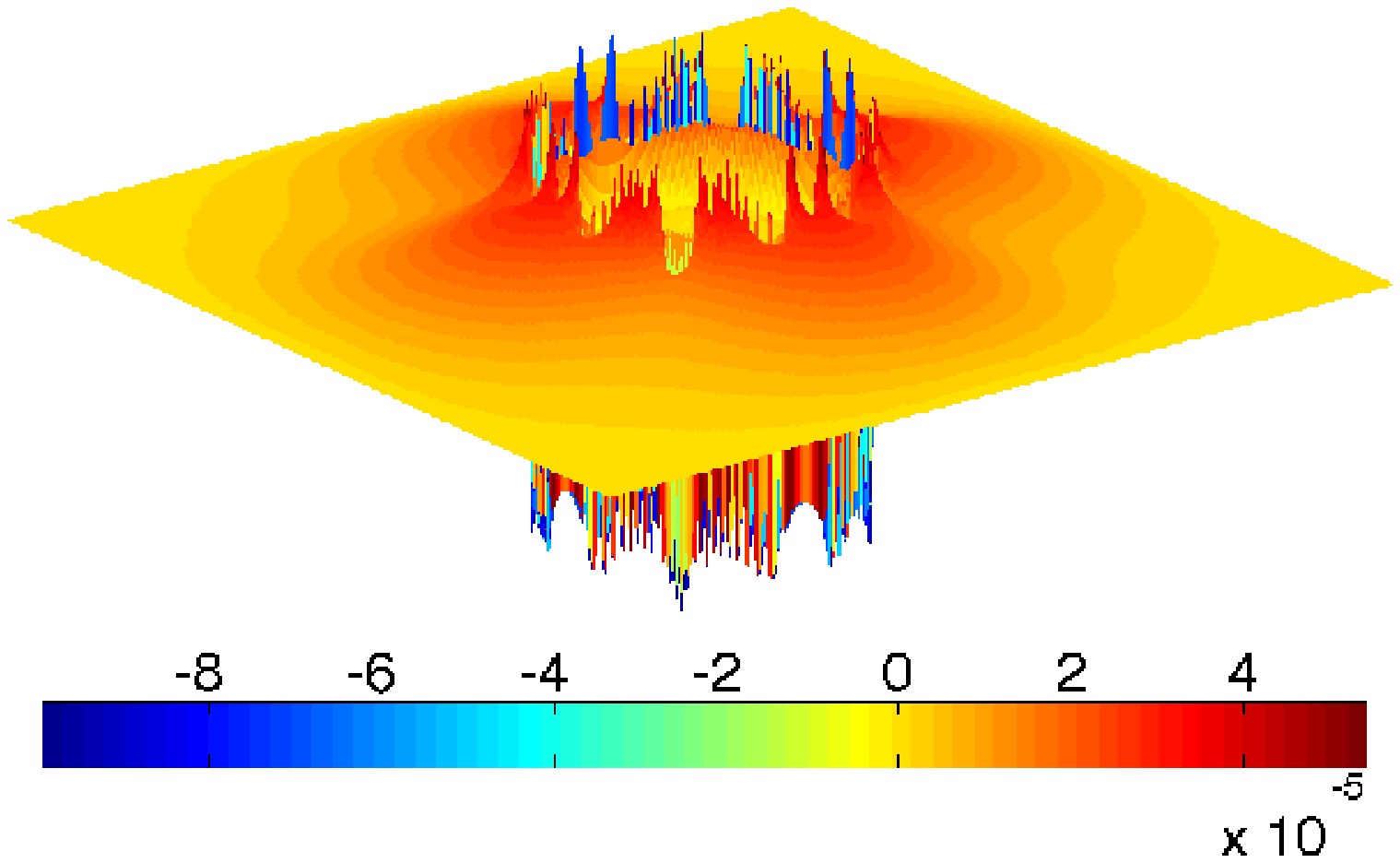}
\caption{The plots above are our numerical solution (left) and the difference with the analytic solution (right) associated with Equation~\eqref{obstacle2dexact}. The grid size is 256 by 256, the parameters are $(\mu,\lambda) =(10/h^2,20.3),$ and $ tol =10^{-6}$.}
\label{fig:obstacle2dexact}
\end{figure}

\begin{figure}[t!]
\centering
\includegraphics[width = 3.2 in]{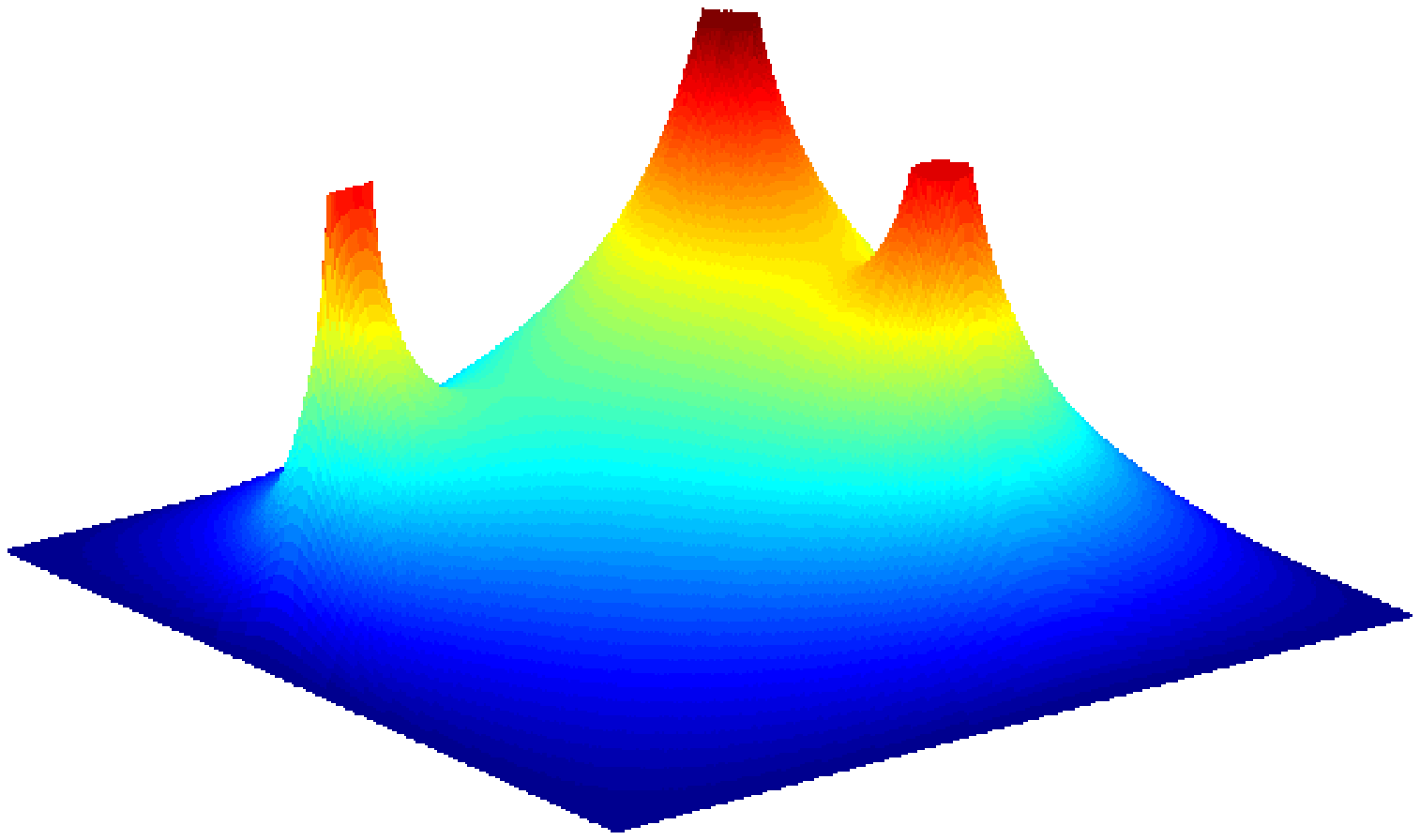}
\includegraphics[width = 1.8 in]{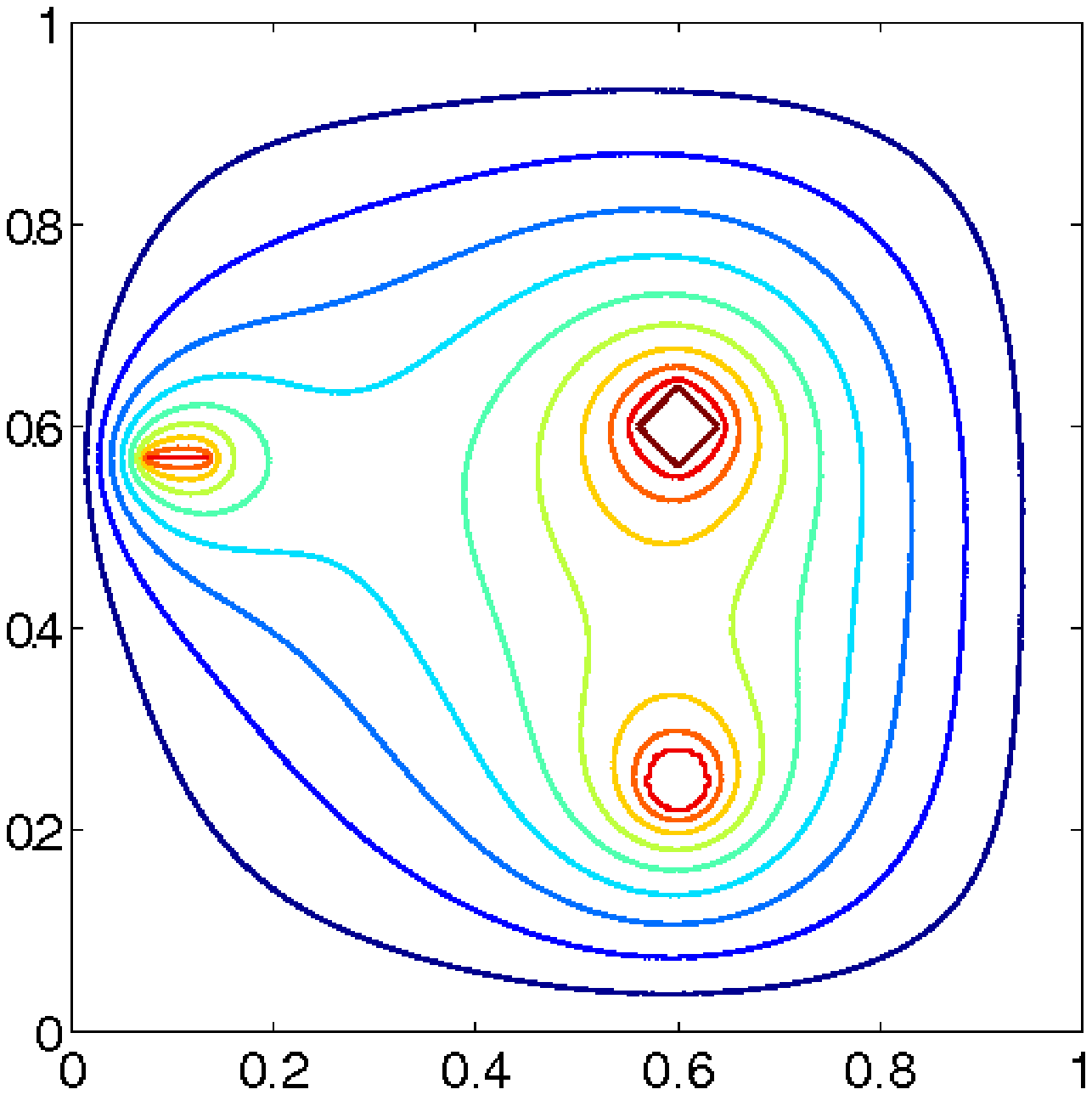}
\caption{The plots above are our numerical solution (left) and its level curves (right) of the obstacle problem associated with Equation ~\eqref{obstacle2d1}. The grid size is 256 by 256, the parameters are $(\mu,\lambda) = (6.5\times 10^5,1.3\times 10^4)$, and $tol = 5\times 10^{-4}$.}
\label{Obstacle2d}
\end{figure}

Next, to examine the behavior of non-smooth obstacles, we consider:
\begin{equation}
\varphi_3 (x,y) = \begin{cases}
5.0, &\quad\text{for}\quad |x-0.6|+ |y-0.6|<0.04\\
4.5,&\quad \text{for}\quad (x-0.6)^2 + (y-0.25)^2 <0.001\\
4.5,&\quad \text{for}\quad y = 0.57 \text{ \ and \ } 0.075 < x < 0.13 \\
0, &\quad\text{otherwise}
\end{cases}
\label{obstacle2d1}
\end{equation}
which consists of different disjoint shapes inside the domain $[0,1]\times [0,1]$. The numerical result and its level curves are shown in Fig. \ref{Obstacle2d}. One can see that the solution is smooth away from the obstacle and agrees well with the obstacle on its support set.

Finally, in Fig. \ref{Obstacle2dbumps}, an obstacle consisting of two intersect planes with a bump on each plane in the domain $[-1,1]\times [-1,1]$ is examined:
\begin{equation}
\varphi_4 = \min (x+y-2,2x+0.5y-2.5)-2e^{-60(x^2 + y^2)} - 1.5e^{-200((x-0.75)^2 + (y+0.5)^2)}.
\label{obstacle2d2}
\end{equation}
In this case, the solution agrees with the obstacle in a large portion of the domain. The analytic solution is given by the two intersecting planes, which can be seen by the linear level curves.
\begin{figure}[t!]
\centering
\includegraphics[width = 2.3 in]{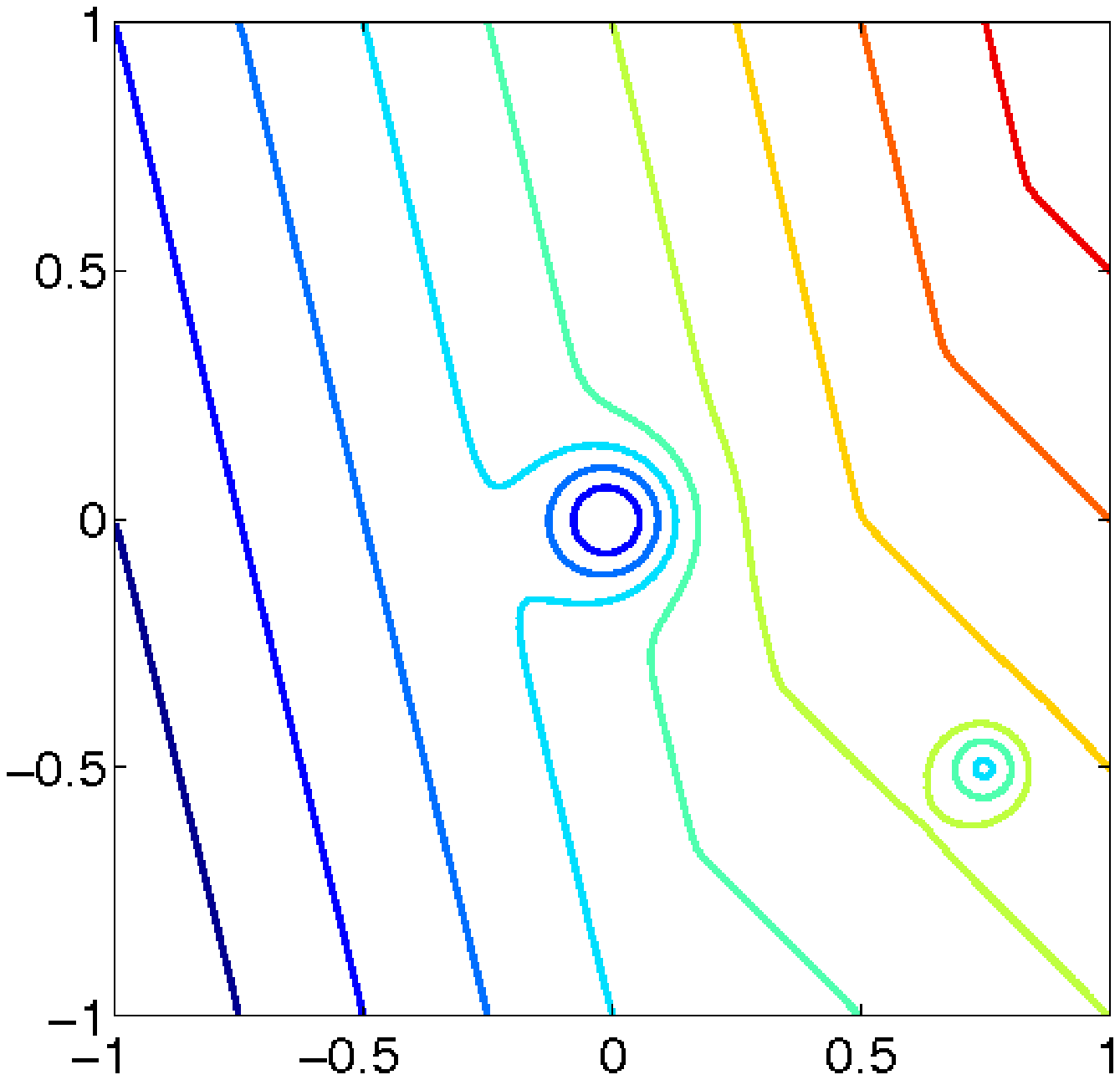}\hspace{0.3 cm}
\includegraphics[width = 2.3 in]{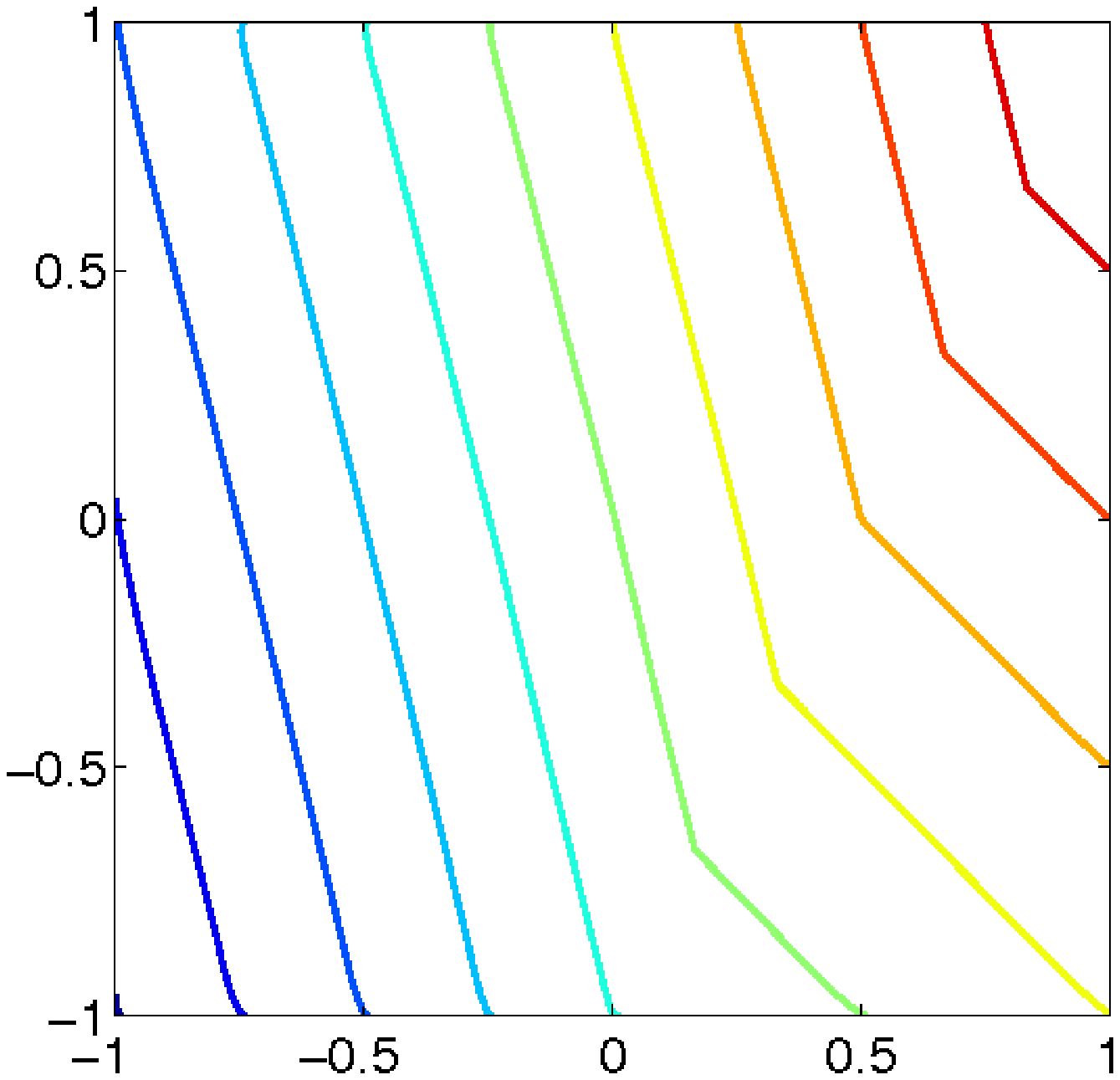}
\caption{The plots above are the level curves of the obstacle (left) and our numerical solution (right) associated with Equation~\eqref{obstacle2d2}. The grid size is 256 by 256, the parameters are $(\mu,\lambda) = (10^5,5\times 10^3)$, and $tol = 5\times 10^{-4}$. }
\label{Obstacle2dbumps}
\end{figure}

\subsection{Nonlinear Obstacle}
\begin{figure}[t!]
\centering
\includegraphics[width = 2.3 in]{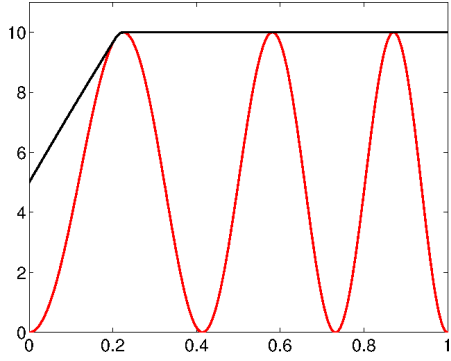}
\caption{The red curve is the obstacle and the black one is the numerical solution of the nonlinear obstacle problem (Equation~\eqref{nonlineareq}). The grid size is 512, $\lambda = 5.3, \mu = 1.1\times 10^3, \tau = 1/L = h^2/2$.}
\label{obstaclenonlinear}
\end{figure}

We would like to show that the methodology here can be easily applied to nonlinear problems, so as a proof of concept we minimize the surface tension:
\begin{align*}
\min\limits_{v\geq \varphi}\int_{0}^1\sqrt{1+|\nabla v|^2}\, dx,
\end{align*}
which is the energy associated with the classical model of stretching an elastic membrane over a fixed obstacle. The obstacle $\varphi$ is given by the oscillatory function:
\begin{equation}
\varphi = 10\sin^2(\pi (x+1)^2),\quad x\in [0,1].
\label{nonlineareq}
\end{equation}
The boundary data for this example is taken to be $u(0) = 5$ and $u(1) = 10$. The numerical solution is linear away from the contact set as can be seen in Fig. \ref{obstaclenonlinear}. 

\subsection{Two-phase membrane problem}
We examine the two-phase membrane problem:
\begin{equation*}
\min_u\int\frac{1}{2}|\nabla u|^2 + \mu_1 u_+ - \mu_2 u_-\,dx,
\end{equation*}
with different sets of $(\mu_1,\mu_2)$ and boundary conditions. First, in the symmetric case, we consider the following 1D problem:
\begin{equation}
u'' = 8\chi_{\{u>0\}} - 8\chi_{\{u<0\}} \quad \text{with}\quad
u(1) = 1, \quad u(-1) = -1,
\label{twophase1d1}
\end{equation}
whose analytic solution is given by:
\begin{equation*}
u(x) = \begin{cases} 
		-4x^2 -4x-1&\quad\text{for}\quad -1\leq x\leq -0.5\\
		0&\quad \text{for}\quad -0.5\leq x\leq 0.5\\
		4x^2 -4x +1&\quad\text{for}\quad \quad 0.5\leq x\leq 1.
	 \end{cases}		
 \end{equation*}
In Fig. \ref{TwoPhase1d1} (left), we plot our numerical solution at the third iteration and the final state. Within a few iterations, our numerical method is able to locate the correct zero set. The error versus the number of iterations is shown in Fig. \ref{TwoPhase1d1} (right), and converges nearly exponentially. For comparison of the numerical results, see \cite{bozorgnia2011numerical}.
\begin{figure}[t!]
\centering
\includegraphics[width = 2.5 in]{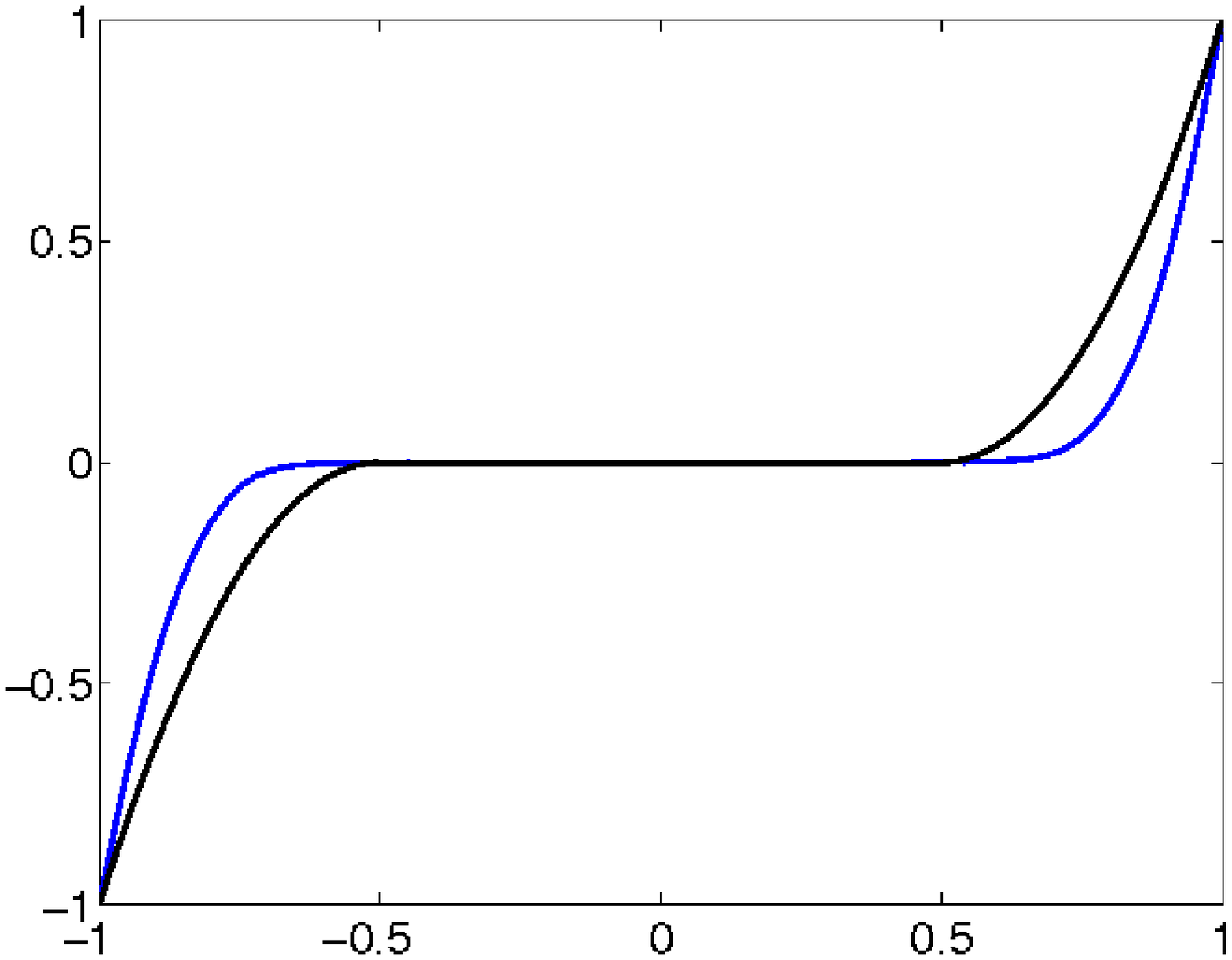}
\includegraphics[width = 2.5 in]{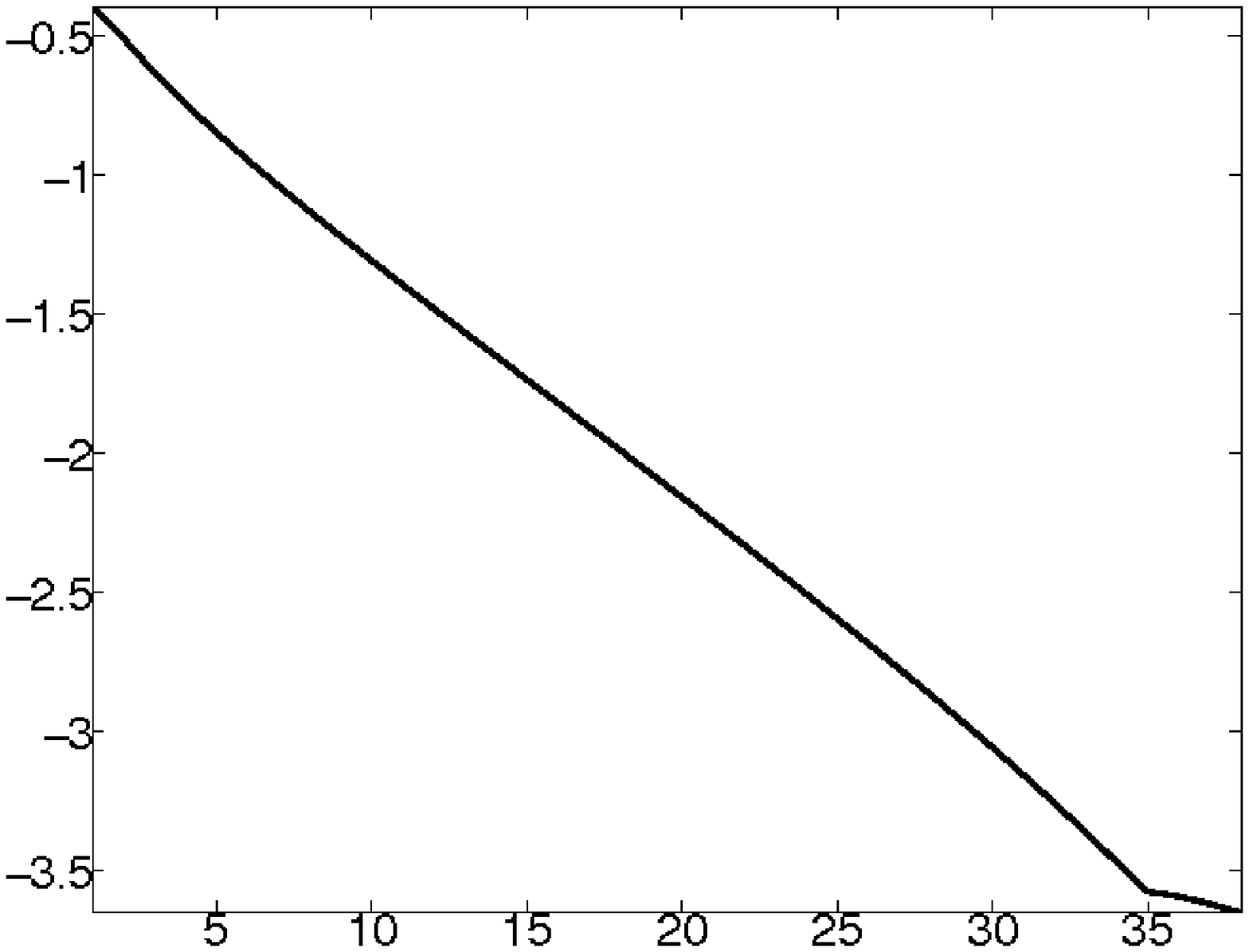}
\caption{Left: The black curve is the final numerical solution, the blue one is the numerical solution after 3 iterations of the two-phase membrane associated with Equation~\eqref{twophase1d1}. The grid size is 512, $\lambda = 204.8$ and $tol = 5\times 10^{-5}$. Right: The log error in $L^\infty$-norm between the numerical and the analytic solutions.} 
\label{TwoPhase1d1}
\end{figure}
\begin{figure}[t!]
\centering
\includegraphics[width = 2.5 in]{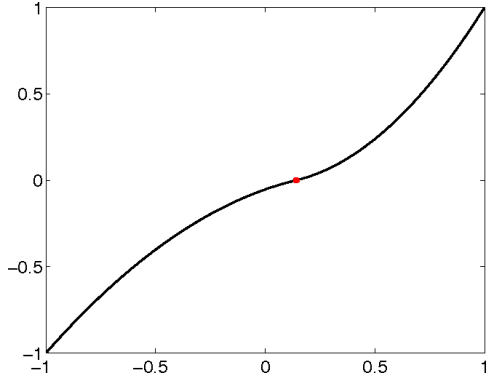}
\caption{The plot above is our numerical solution of the two-phase membrane associated with Equation~\eqref{twophase1d2}. The free boundary point is marked in red and is located at $x\approx 0.141$. The grid size is $2^{12}$, $\lambda = 3072$, and $tol =5\times 10^{-7}$.}
\label{TwoPhase1d2}
\end{figure}

Next we consider a non-symmetric equation:
\begin{equation}
u'' = 2\chi_{\{u>0\}} - \chi_{\{u<0\}} \quad \text{with}\quad
u(1) = 1, \quad u(-1) = -1,
\label{twophase1d2}
\end{equation}
The calculated free boundary is at the point $x\approx 0.141$, which was also observed in \cite{shahgholian2007two} (see Fig. \ref{TwoPhase1d2}).

For an example in 2D, we set $\mu_1 = \mu_2=1$ with Dirichlet  boundary condition $g$ given by:
\begin{equation}
g(x,y) = \begin{cases}
(1-x)^2/4\quad & -1\leq x\leq 1\,\, \text{and } \,\, y=1\\
-(1-x)^2/4\quad & -1\leq x\leq 1\,\, \text{and } \,\, y=-1\\
y^2\quad & 0\leq y\leq 1\,\, \text{and } \,\, x=-1\\
-y^2\quad & -1\leq y\leq 0\,\, \text{and } \,\, x=-1\\
0\quad & -1\leq y\leq 1\,\, \text{and } \,\, x=1.
\end{cases}
\label{branchingpoint11}
\end{equation}
In this case, the zero set has non-zero measure, see Fig. \ref{TwoPhase2d} (right). The boundary between the regions $\{u>0\}$, $\{u<0\}$ and $\{u=0\}$ contains a branching point, which we are able to resolve numerically. 
\begin{figure}[t!]
\centering
\includegraphics[width = 2.5 in]{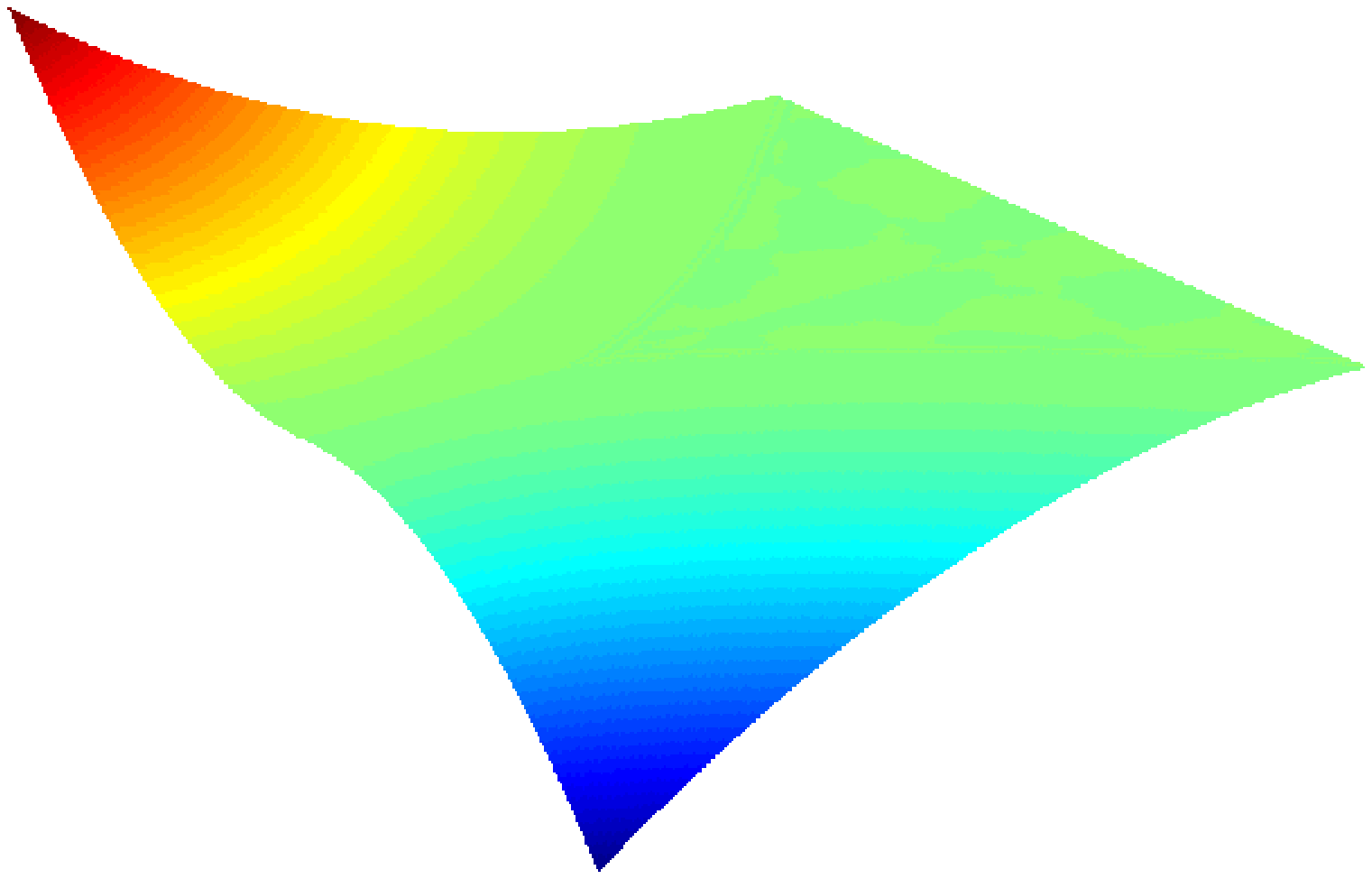}
\includegraphics[width = 2.5 in]{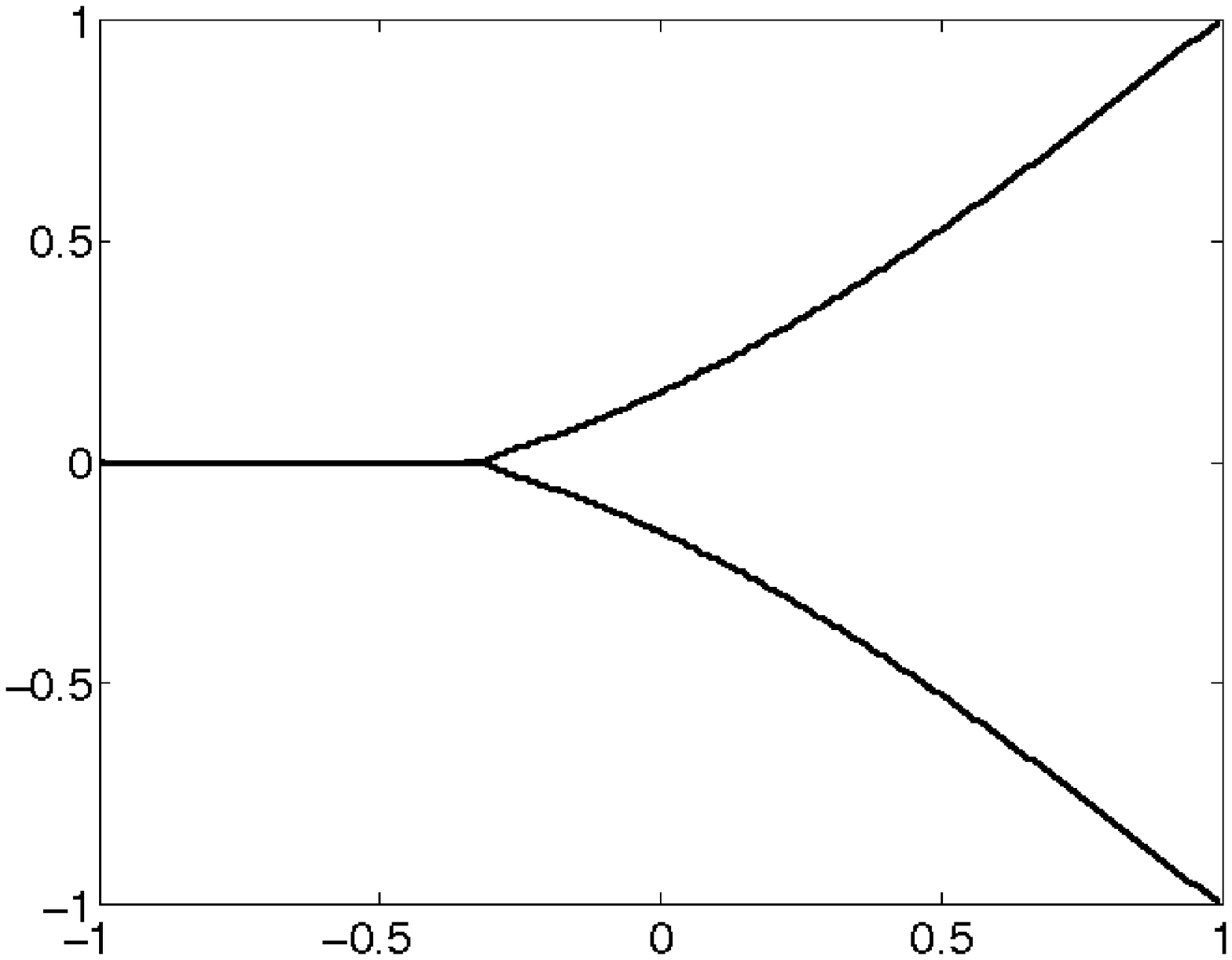}
\caption{Left: Our numerical solution associated with Equation~\eqref{branchingpoint11}. Right: the boundaries between the regions $\{u>0\}$ (top), $\{u<0\}$ (bottom) and $\{u=0\}$. The grid size is 256 by 256, $\lambda = 100.0$ and $tol = 10^{-6} $. }
\label{TwoPhase2d}
\end{figure}

\subsection{Hele-Shaw}
We present three examples of the Hele-Shaw problem:
\begin{equation*}
\min_u  \int \frac{1}{2}|\nabla u|^2+\gamma_1(\varphi - u)_+-\gamma_2 (t\chi_K - u)_- \, dx,
\end{equation*}
with different sets of $(K,\Omega_0)$. The parameters are fixed at $\gamma_1 = \gamma_2 = 1.5\times 10^4,  \lambda_1 = \lambda_2 = 150$. The free boundary starts moving from $\Omega_0$. 
\begin{table}[t!]
\begin{center}
\rowcolors{1}{lightgray}{lightgray}
 \begin{tabular}{|c|c|c|c|c|}\hline
 Grid Size & $128\times 128$ & $256 \times 256$ & $ 512 \times 512$ & $1024 \times 1024$\\ \hline
Error (radius) & 0.0238 & 0.0124 & 0.0083 & 0.0044\\
\hline
\end{tabular}
\caption{The error between the radius of the free boundary of our numerical solution and the exact radius associated with $(K,\Omega_0)$ defined in Equation~\eqref{eqn:Hele-Shawcircles} at time $t = 0.25$. The parameters are fixed at $\gamma_1=\gamma_2= 1.5\times 10^4, \lambda_1 = \lambda_2 = 150, \, tol = 10^{-6}$. The convergence rate is approximately $\mathcal{O}(h^{0.8})$.}
\end{center}
\label{Table1}
\end{table}

To validate our numerical scheme, in the first example we compare the boundary of our numerical solution and that of the ground truth solution. In particular, when both $K$ and $\Omega_0$ are circles centered at the origin:
\begin{equation}
K = \{(x,y)\in [-5,5]^2 \mid x^2 + y^2\leq 1\}, \quad \Omega_0 = \{(x,y)\in [-5,5]^2 \mid x^2 + y^2\leq 2\},
\label{eqn:Hele-Shawcircles}
\end{equation}
 the free boundary remains a circle centered at the origin for all time. Thus the radius of $\Omega$, $R_{exact}$, can be calculated explicitly. In Table 1, we compute the error between the radius of the free boundary of our numerical solution and the analytic solution at time $t = 0.25$ using different grid sizes. The experimental error in the radius is about $\mathcal{O}(h^{0.8})$, which is expected for a low dimensional structure.

Next, we present two numerical results for more complicated cases of $(K,\Omega_0)$. In Fig. \ref{HeleShawr} (left), the free boundary $\partial \Omega$ is pinned at the two acute vertices along $\partial \Omega_0$. As expected, the free boundary opens up to right angles then smooth out and move away from $\Omega_0$. For more details on this short time behavior as well as singularities in the Hele-Shaw model see  \cite{jerison2005one, tian1998singularities,nie2001singularities}. Finally, in Fig. \ref{HeleShawr} (right), we take the boundary of $\Omega_0$ to be smooth but concave. The free boundary  moves away from the initial state and begins to smooth out. 
\begin{figure}[t!]
\centering
\includegraphics[width = 2.5 in]{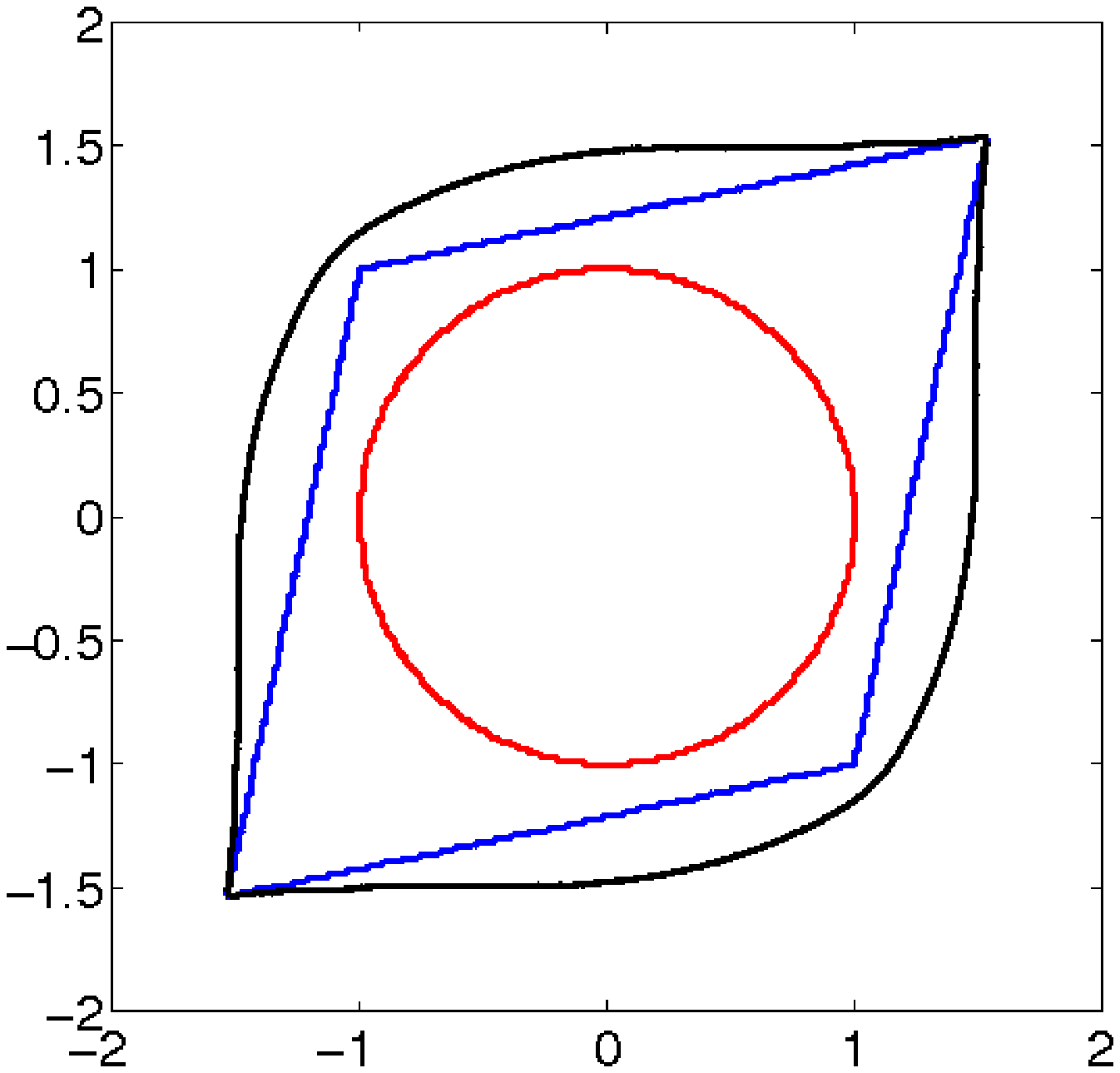}
\includegraphics[width = 2.5 in]{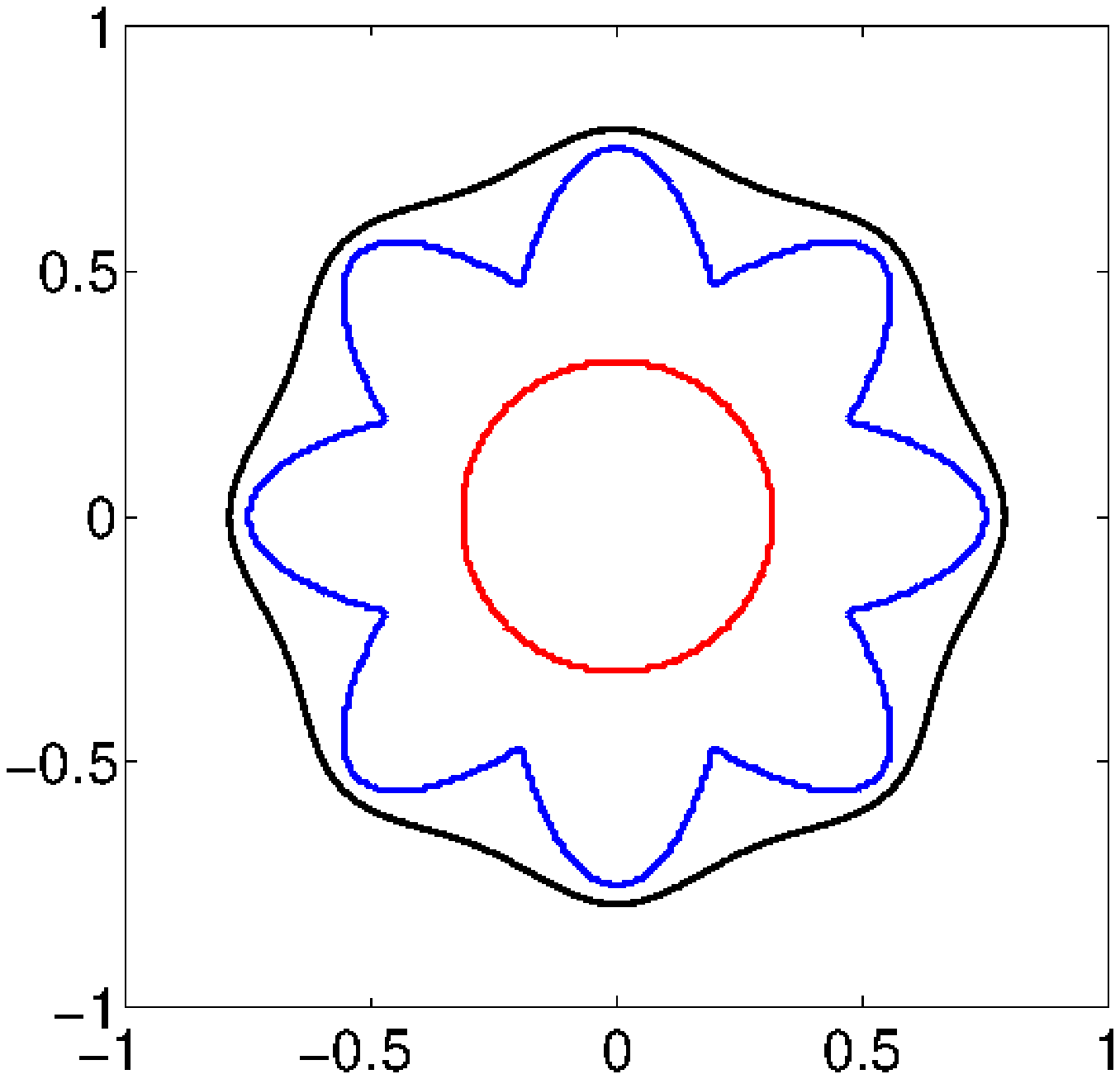}
\caption{Boundaries of the sets \textcolor{red}{K}, \textcolor{blue}{$\Omega_0$} and \textcolor{black}{the \bf  free boundary} of the Hele-Shaw problem. The grid size is 256 by 256, $ \gamma_1=\gamma_2 = 1.5\times 10^4, \lambda_1 = \lambda_2 = 150, tol = 10^{-5},$ and time $t=0.1$ and $t=0.06$, respectively. } 
\label{HeleShawr}
\end{figure}

 \medskip

\section{Conclusion}\label{sec:conclusion}

Using an $L^1$-penalty method, we are able to construct an unconstrained problem whose solutions correspond exactly to those of the obstacle problems. We provide a lower bound on the value of the penalty parameter and use this to guide our numerical calculations.  We present several experiment results showing the applicability of our method to various physical problems.

\section*{Acknowledgement}
The authors would like to thank Inwon Kim for her helpful discussions. G. Tran is supported by UC Lab 443948-B1-69763 and Keck Funds 449041-PW-58414. H. Schaeffer is supported by NSF 1303892 and University of California Presidents Postdoctoral Fellowship Program. W. Feldman is supported by NSF DMS 1300445. S. Osher is supported by ONR Grant N00014-11-1-719. 
\appendix

\bibliographystyle{plain}

\end{document}